\documentclass[letterpaper, 11pt,  reqno]{amsart}

\usepackage{amsmath,amssymb,amscd,amsthm,amsxtra, esint}

\usepackage[implicit=true]{hyperref}

\allowdisplaybreaks[2]

\usepackage{color}

\definecolor{gr}{rgb}   {0.,   0.69,   0.23 }
\definecolor{bl}{rgb}   {0.,   0.5,   1. }
\definecolor{mg}{rgb}   {0.85,  0.,    0.85}
\definecolor{yl}{rgb}   {0.8,  0.7,   0.}
\definecolor{or}{rgb}  {0.7,0.2,0.2}

\newtheorem{theorem}{Theorem} [section]

\newtheorem{lemma}[theorem]{Lemma}
\newtheorem{proposition}[theorem]{Proposition}
\newtheorem{remark}[theorem]{Remark}

\newtheorem{definition}[theorem]{Definition}
\newtheorem{corollary}[theorem]{Corollary}

\newtheorem*{acknowledgment}{Acknowledgments}

\DeclareMathOperator*{\intt}{\int}

\DeclareMathOperator*{\supp}{supp}

\newcommand{\1}{\hspace{0.5mm}\text{I}\hspace{0.5mm}}
\newcommand{\II}{\text{I \hspace{-2.8mm} I} }
\newcommand{\III}{\text{I \hspace{-2.9mm} I \hspace{-2.9mm} I}}

\newcommand{\I}{\mathcal{I}}

\newcommand{\noi}{\noindent}
\newcommand{\Z}{\mathbb{Z}}
\newcommand{\R}{\mathbb{R}}
\newcommand{\C}{\mathbb{C}}
\newcommand{\T}{\mathbb{T}}

\let\Re=\undefined\DeclareMathOperator*{\Re}{Re}
\let\Im=\undefined\DeclareMathOperator*{\Im}{Im}

\let\P= \undefined
\newcommand{\P}{\mathbf{P}}

\newcommand{\E}{\mathcal{E}}

\newcommand{\N}{\mathcal{N}}
\newcommand{\NB}{\mathbb{N}}

\newcommand{\al}{\alpha}

\newcommand{\dl}{\delta}

\newcommand{\nb}{\nabla}

\newcommand{\Dl}{\Delta}
\newcommand{\eps}{\varepsilon}

\newcommand{\g}{\gamma}
\newcommand{\G}{\Gamma}
\newcommand{\ld}{\lambda}

\newcommand{\s}{\sigma}

\newcommand{\ft}{\widehat}

\newcommand{\wt}{\widetilde}
\newcommand{\cj}{\overline}

\newcommand{\dt}{\partial_t}
\newcommand{\dd}{\partial}

\renewcommand{\l}{\ell}
\renewcommand{\o}{\omega}
\renewcommand{\O}{\Omega}

\newcommand{\les}{\lesssim}
\newcommand{\ges}{\gtrsim}

\newcommand{\jb}[1]
{\langle #1 \rangle}

\newcommand{\ind}{\mathbf 1}

\newcommand{\lr}[1]{\langle #1 \rangle}

\DeclareMathOperator{\Id}{Id}

\numberwithin{equation}{section}
\numberwithin{theorem}{section}

\newcommand{\too}{\longrightarrow}

\begin{document}
\baselineskip = 15pt

\title[Probabilistic well-posedness of NLS on $\R^d$, $d = 5, 6$] 
{On the probabilistic  well-posedness of the  nonlinear Schr\"{o}dinger equations
with non-algebraic nonlinearities}

\author[T. Oh]{Tadahiro Oh}
\address{
Tadahiro Oh\\
School of Mathematics\\
The University of Edinburgh\\
and The Maxwell Institute for the Mathematical Sciences\\
James Clerk Maxwell Building\\
The King's Buildings\\
 Peter Guthrie Tait Road\\
Edinburgh\\ 
EH9 3FD\\United Kingdom} 

\email{hiro.oh@ed.ac.uk}

\author[M. Okamoto]{Mamoru Okamoto}
\address{Mamoru Okamoto\\
Division of Mathematics and Physics, Faculty of Engineering, Shinshu University, 4-17-1 Wakasato, Nagano City 380-8553, Japan}
\email{m\_okamoto@shinshu-u.ac.jp}

\author[O. Pocovnicu]{Oana Pocovnicu}
\address{
Oana Pocovnicu\\
Department of Mathematics\\
Heriot-Watt University and The Maxwell Institute for the Mathematical Sciences\\
 Edinburgh\\ EH14 4AS\\ United Kingdom
}

\email{o.pocovnicu@hw.ac.uk}



\begin{abstract}
We consider the Cauchy problem for 
the nonlinear Schr\"odinger equations (NLS) with non-algebraic nonlinearities on the Euclidean space.
In particular, we study 
the   energy-critical NLS  on $\R^d$, $d=5,6$, 
and energy-critical NLS without gauge invariance
and  prove that they are almost surely locally well-posed with respect to randomized initial data below the energy space.
We also study the long time behavior of  solutions to these equations:
(i) 
we  prove almost sure global well-posedness
of the (standard) energy-critical NLS  on $\R^d$, $d = 5, 6$, in the defocusing case,
and (ii) 
we present a probabilistic construction of finite time blowup solutions
 to the energy-critical NLS without gauge invariance below the energy space.

\end{abstract}

%
\maketitle
\tableofcontents

\baselineskip = 14pt

\section{Introduction}\label{SEC:1}

\subsection{Nonlinear Schr\"odinger equations}
We consider the Cauchy problem for the following energy-critical nonlinear Schr\"{o}dinger equation (NLS)
on $\R^d$, $d = 5, 6$:
\begin{equation}
\label{NLS}
\begin{cases}
 i \dt u + \Dl  u = \pm |u|^{\frac{4}{d-2}} u  \\
 u|_{t = 0} = \phi,
\end{cases}
\qquad (t, x) \in \R\times \R^d.
\end{equation}

\noi
This equation enjoys the following dilation symmetry:
\[
u(t,x) \ \longmapsto \ u_{\mu} (t,x) := \mu^{\frac{d-2}{2}} u ( \mu^2 t  , \mu x )
\]
for $\mu >0$.
This dilation symmetry preserves  the $\dot H^1$-norm 
of the initial data $\phi$, thus inducing the scaling critical Sobolev regularity  $s_\text{crit} = 1$.
Moreover,  the energy (= Hamiltonian) of a solution $u$
remains invariant under this dilation symmetry.
For this reason, we refer to~\eqref{NLS} as energy-critical 
and $\dot H^1(\R^d)$ as the energy space.

The Cauchy problem \eqref{NLS} in a general dimension has been 
at the core of the study of dispersive equations for several decades
and has been studied extensively.
In particular, for $d \geq  5$, it is known that \eqref{NLS} is (i) locally well-posed 
in the energy space \cite{CW}
and (ii) globally well-posed in the defocusing case \cite{V}
and also in the focusing case under some assumption on the (kinetic) energy \cite{KV}.
On the other hand, 
\eqref{NLS} is known to be ill-posed in $H^s(\R^d)$, $s < s_\text{crit} = 1$, 
in the sense of norm inflation \cite{CCT}; 
 there exists a sequence $\{ u_n\}_{n\in \NB}$ of (smooth) solutions to \eqref{NLS}
and $\{t_n\}_{n \in \NB} \subset \R_+$ such that 
$\|u_n(0) \|_{H^s} < \frac 1n$
but $\|u_n(t_n) \|_{H^s} > n$
with $t_n < \frac 1n$.
This in particular shows that the solution map to \eqref{NLS}
can not be extended to be a continuous map on $H^s(\R^d)$, $s < 1$,
thus violating one of the important criteria for well-posedness.

Despite the ill-posedness below the energy space, 
one may still hope to construct unique local-in-time solutions
in a probabilistic manner, thus establishing almost sure local well-posedness
in some suitable sense.
Such an approach first appeared in the work 
by McKean~\cite{McKean} and Bourgain \cite{BO96}
in the study of invariant Gibbs measures for the cubic NLS on $\T^d$, $d = 1, 2$.
In particular, they established almost sure local well-posedness
with respect to particular random initial data.\footnote{These local-in-time solutions
were then extended globally in time by invariance
of the Gibbs measures.
In the following, however, we do not use any invariant measure.}
This random initial data in \cite{McKean, BO96}
can be viewed as a randomization of the Fourier coefficients
of a particular function (basically the antiderivative of the Dirac delta function)
via the multiplication by independent Gaussian random variables.
Such  randomization of the Fourier series is classical
and  well studied \cite{PZ, Kahane}.
In \cite{BT1}, 
Burq-Tzvetkov elaborated this idea further.
In particular, in the context of the cubic nonlinear wave equation (NLW)
on a three dimensional compact Riemannian manifold, 
they considered a randomization via the Fourier series expansion as above
for  {\it any} rough initial condition below the scaling critical 
Sobolev regularity
and established almost sure local well-posedness with respect to the randomization.
Such randomization via the Fourier series expansion
is natural 
on compact domains and more generally
in situations where the associated elliptic operators have discrete spectra
\cite{Tho09, DC, CO}.

Our main focus is to study NLS \eqref{NLS} on the Euclidean space $\R^d$.
In this setting, the randomization via the Fourier series expansion does not quite work
as the frequency space $\R^d_\xi$ is not discrete.
We instead consider a randomization associated to the Wiener decomposition 
$\R^d_\xi = \bigcup_{n \in \Z^d} (n+ (-\frac 12, \frac12]^d)$.
See  \cite{ZF, LM, BOP1, BOP2, HO16}.
Let $\psi \in \mathcal S(\R ^d)$ satisfy
\[
\supp \psi \subset [-1,1]^d 
\qquad \text{and}\qquad \sum _{n \in \mathbb{Z}^d} \psi (\xi -n) =1 \quad \text{for any $\xi \in \R ^d$}.
\]

\noi
Then, given a function $\phi$ on $\R^d$, we have
\[ \phi =  \sum _{n \in \mathbb{Z}^d} \psi (D-n) \phi.\]

\noi
This replaces the role of the Fourier series expansion on compact domains.
We then define the Wiener randomization of $\phi$ by
\begin{equation} \label{rand}
\phi ^{\omega} := \sum _{n \in \mathbb{Z}^d} g_n (\omega ) \psi (D-n) \phi,
\end{equation}

\noi
where $\{ g_n \}$ is a sequence of independent mean zero complex-valued random variables 
on a probability space $(\Omega , \mathcal{F} ,P)$.
In the following, we assume that 
 the real and imaginary parts of $g_n$ are independent and endowed with probability distributions $\mu _n^{(1)}$ and $\mu _n^{(2)}$, satisfying the following exponential moment bound:
\[
\int _{\R} e^{\kappa x} d \mu _{n}^{(j)} (x) \le e^{c \kappa ^2}
\]
for all $\kappa \in \R$, $n \in \mathbb{Z}^d$, $j=1,2$.
This condition is satisfied by the standard complex-valued Gaussian random variables and the standard Bernoulli random variables.

On the one hand, the randomization does not improve differentiability
just like the randomization via the Fourier series expansion
\cite{BT1, AT}.
On the other hand, it 
improves integrability
as for the classical random Fourier series \cite{PZ, Kahane}.
From this point of view, the randomization makes the problem {\it subcritical} in some sense,
at least for local-in-time problems.

In the following, we study the Cauchy problem \eqref{NLS} with
random initial data given by the Wiener randomization $\phi^\o$ of a given function $\phi \in H^s(\R^d)$, $d = 5, 6$.
In view of the deterministic well-posedness result for $s \geq 1$, 
we only consider $s < s_\text{crit} = 1$.

\begin{theorem} \label{THM:LWP1}
Let $d=5,6$ and  $1-\frac{1}{d}<s<1$.
Given $\phi \in H^s (\R ^d)$, let $\phi^{\omega}$ be its Wiener randomization defined in \eqref{rand}.
Then, 
the Cauchy problem \eqref{NLS} is almost surely locally well-posed
with respect to the random initial data $\phi^\o$.

More precisely, there exist $C, c, \gamma >0$ such that for each $0<T\ll1$, there exists $\Omega _{T} \subset \Omega$ with $P(\Omega _T^c) \leq  C \exp \left( - \frac{c}{T^{\gamma} \| \phi \|_{H^s}^2} \right)$
such that  for each $\omega \in \Omega_T$, 
there exists a unique solution $u=u^\o \in C([-T,T]; H^s(\R^d))$ to \eqref{NLS} with $u|_{t = 0} = \phi ^{\omega}$ in the class
\begin{align*}
S(t)\phi^\o + X^1_T \subset
 S(t) \phi ^{\omega} + C([-T,T] ; H^1(\R^d)) \subset C([-T,T] ; H^s(\R^d)),
\end{align*}

\noi
where $S(t) = e^{it\Dl}$ and $X^1_T$ is defined in Section \ref{SEC:space} below.
\end{theorem}

Almost sure local well-posedness  with respect to the Wiener randomization
has been studied in the context of  the cubic NLS 
and the quintic NLS on $\R^d$ \cite{BOP1, BOP2, Bre}
which are energy-critical in dimensions 4 and 3, respectively.
Note that when $d = 5, 6$, the energy-critical nonlinearity $ |u|^{\frac{4}{d-2}} u$ is no longer algebraic,
presenting a new difficulty in applying the argument in \cite{BOP1, BOP2, Bre}.

Let $z(t) = z^{\omega} (t) := S(t) \phi ^{\omega}$ denote the random linear solution
with $\phi^\o$ as initial data. 
If $u$ is a solution to \eqref{NLS}, then the residual term 
$v := u - z$ satisfies the following perturbed NLS:
\begin{equation}
\begin{cases}
 i\partial _t + \Delta  v =  \mathcal{N} (v+z^\o) \\
 v|_{t= 0} =0 ,
\end{cases}
\label{NLS1a}
\end{equation}

\noi
where  $\mathcal{N} (u) = \pm |u|^{\frac{4}{d-2}} u$.
In terms of the Duhamel formulation, \eqref{NLS1a} reads as
\begin{equation}
v(t) =- i\int_0^t S(t - t') \N(v+z^\o) (t') dt'.
\label{NLS1b}
\end{equation}

\noi
Then, the main objective is to solve the fixed point problem \eqref{NLS1b}.\footnote{In the field
of stochastic parabolic PDEs, this change of viewpoint
and solving the fixed point problem for the residual term $v$ is called
the Da Prato-Debussche trick \cite{DPD, DPD2}.
In the context of deterministic dispersive PDEs with random initial data, 
this goes back to the work by McKean \cite{McKean} and Bourgain \cite{BO96},
which precedes \cite{DPD, DPD2}.}
In fact, the first and third authors (with B\'enyi) \cite{BOP1, BOP2} studied
this problem for the residual term $v$ in the context of  the cubic NLS on $\R^d$ 
by carrying out case-by-case analysis
and estimating terms of the form $v \cj v v$, $v \cj v z$, $v \cj z z$, etc.
In \cite{Bre}, Brereton carried out similar analysis for the quintic NLS on $\R^d$.
Such case-by-case analysis is possible only for algebraic, i.e.~smooth, nonlinearities
and thus is not applicable to our problem at hand.
In this paper, we adjust the analysis from~\cite{BOP2}
in order to  handle  non-algebraic  nonlinearities.
Moreover, our analysis in this paper is simpler than that in \cite{BOP1, BOP2}
in the sense that we avoid thorough case-by-case analysis.
There is, however, a price to pay:
(i) While our approach for non-algebraic nonlinearities in this paper can be applied to 
the energy-critical cubic NLS on $\R^4$,
this would yield a worse regularity range   $s \in (\frac 34, 1)$
than the regularity range $s \in (\frac 35, 1)$ obtained  in \cite{BOP2}. 
This is due to the fact that 
we adjust our calculation to a non-smooth  nonlinearity.
(ii)~The constants in the nonlinear estimates in Section \ref{SEC:nonlin}
depend on the local existence time $T>0$ (see Proposition \ref{PROP:nonlin1} below).
In particular, Theorem \ref{THM:LWP1} is not accompanied
by almost sure small data global well-posedness and scattering.
This is in sharp contrast with the situation for the cubic nonlinearity considered in \cite{BOP2}.

Our main tools for proving  Theorem \ref{THM:LWP1} are similar
to those in \cite{BOP2};
the Fourier restriction norm method adapted to the spaces $V^p$ of functions of bounded $p$-variation
and their pre-duals $U^p$,  the bilinear refinement of the Strichartz estimate,
and the probabilistic Strichartz estimates thanks to the gain of integrability via the Wiener randomization.
In order to avoid the use of fractional derivatives, 
we focus on the energy-critical NLS
and solve the fixed point problem \eqref{NLS1b} in $X^1_T$ at the critical regularity
(for the residual term)
by performing a precise computation.
Namely, it is important that we use this refined version of the Fourier restriction norm method,
since if we were to use the usual  $X^{\s, b}$-spaces introduced in \cite{Bo2}, 
then we would need to study
the problem at the subcritical regularity $\s = 1+ \eps$ as in~\cite{BOP1},
creating a further difficulty.
Moreover, in proving  almost sure global well-posedness of \eqref{NLS}, 
it is essential that we only use the $X^\s_T$-norm, $\s \leq 1$, for the residual part $v$.
See Theorem \ref{THM:GWP} below.

Next, we consider the following energy-critical NLS without gauge invariance 
on $\R^d$, $d = 5, 6$:
\begin{equation}
\label{NLS2}
\begin{cases}
 i \dt u  + \Dl  u = \ld |u|^{\frac{d+2}{d-2}} \\
 u|_{t = 0} = \phi, 
\end{cases}
\end{equation}

\noi
where 
$\ld \in \C \setminus \{0\}$.
As in the case of the standard NLS \eqref{NLS}, 
one can prove local well-posedness of \eqref{NLS2} in $H^s(\R^d)$, $s \geq 1$, 
via the Strichartz estimates.
On the other hand, 
Ikeda-Inui~\cite{II15} showed 
that \eqref{NLS2} is ill-posed in $H^s(\R^d)$ with $s<1$.
More precisely, they proved non-existence of solutions
for rough initial data, satisfying a certain condition.
This ill-posedness result by non-existence is much stronger than 
the norm inflation proved for the standard NLS \eqref{NLS}.
The non-existence result in \cite{II15} studies
a rough initial condition and exhibits a pathological behavior in a direct manner,
while the norm inflation result in~\cite{CCT} is proved by studying the behavior of a sequence
of smooth solutions; in particular it does not say anything about rough solutions. 

\begin{theorem} \label{THM:LWP2}
Let $d=5,6$ and  $1-\frac{1}{d}<s<1$.
Given $\phi \in H^s (\R ^d)$, let $\phi^{\omega}$ be its Wiener randomization defined in \eqref{rand}.
Then, the Cauchy problem \eqref{NLS2} is almost surely locally  well-posed 
with respect to the random initial data $\phi^\o$
in the sense of Theorem \ref{THM:LWP1}.
\end{theorem}

Theorem \ref{THM:LWP2}, in particular, states that
upon the randomization, we can avoid these pathological initial data
constructed in \cite{II15}
for which no solution exists.
Compare this with the ``standard'' almost sure local well-posedness
results such as Theorem~\ref{THM:LWP1} above,
where the only known obstruction to  well-posedness below
a threshold regularity is discontinuity of the solution map.\footnote{Namely, 
the pathological behavior of the standard NLS \eqref{NLS}
below the scaling critical regularity $s_\text{crit} = 1$
is about the solution map (stability under perturbation) and is not about individual solutions (such as existence).
On the contrary, in the case of \eqref{NLS2}, there are individual initial data, 
each of which is responsible for the pathological behavior (non-existence of solutions).}
In this sense, Theorem \ref{THM:LWP2}
provides a more striking role of randomization, overcoming 
the non-existence result below the scaling critical regularity, 
and it seems that Theorem \ref{THM:LWP2} is the first such result.

The proof of Theorem~\ref{THM:LWP2} follows the same lines as that of Theorem \ref{THM:LWP1}.
When $d = 6$,  the nonlinearity  $|u|^2 = u\cj u$
in \eqref{NLS2} is algebraic.
Hence, one may also perform case-by-case analysis as in \cite{BOP2}.
We, however, do not pursue this direction 
since our purpose is to present a unified approach to the problem.

Next, let us state 
an almost sure local well-posedness result with slightly more general initial data.
Fix $\phi \in H^s(\R^d)\setminus H^1(\R^d)$.
Then, we consider the following Cauchy problem for given $v_0 \in H^1(\R^d)$:
\begin{equation}
\label{NLS3}
\begin{cases}
 i \dt u + \Dl  u = \N(u)\\
 u|_{t = 0} = v_0+ \phi^\o,
\end{cases}
\end{equation}

\noi
where 
$ \N(u) = \pm |u|^\frac{4}{d-2} u$ or $\ld |u|^\frac{d+2}{d-2}$
and $\phi^\o$ is the Wiener randomization of $\phi$.
Then, 
as a corollary to (the proof of) Theorems \ref{THM:LWP1} and \ref{THM:LWP2},
we have the following proposition.

\begin{proposition}\label{PROP:LWP3}
Let $d=5,6$ and  $1-\frac{1}{d}<s<1$.
Given $\phi \in H^s (\R ^d)$, let $\phi^{\omega}$ be its Wiener randomization defined in \eqref{rand}.
Then, given $v_0 \in H^1(\R^d)$, 
the Cauchy problem~\eqref{NLS3} is almost surely locally well-posed
with respect to  the Wiener randomization $\phi^\o$,
where 
the (random) local existence time $T = T_\o$
is assumed to be sufficiently small,
depending on the deterministic part $v_0$ of the initial data.
Moreover, the following blowup alternative holds;
let $T^* = T^*(\o, v_0)$ be the  forward maximal time of existence.
Then, either 
\begin{align}
T^* = \infty\qquad \text{or} \qquad \lim_{T\to T^*} \| u -S(t) \phi^\o \|_{L_t^{q_d}([0, T);W_x^{1, r_d})} = \infty,
\label{blowup}
\end{align}

\noi
where $(q_d,r_d)$ is a particular admissible pair given by 
\begin{align}
(q_d,r_d) := \big( \tfrac{2d}{d-2}, \tfrac{2d^2}{d^2-2d+4} \big).
\label{qr0}
\end{align}

\end{proposition}

Namely, this is an almost sure local well-posedness result 
with the initial data of the form:~``a {\it fixed} smooth deterministic function 
$+$
 a rough random perturbation''.
See, for example,~\cite{OQ}.
The proof of Proposition \ref{PROP:LWP3}
is based on 
 studying the equation for the residual term $v = u - z^\o$ as above:
\begin{equation}
\begin{cases}
 i\partial _t v + \Delta  v =  \mathcal{N} (v+z^\o) \\
 v|_{t= 0} =v_0 \in H^1(\R^d) , 
\end{cases}
\label{NLS4}
\end{equation}

\noi
where  we now have a non-zero initial condition.
For this fixed point problem, 
the critical nature of the problem appears through  the deterministic initial condition  $v_0$.
In particular, the local existence time $T = T(v_0)$ depends on the profile of the (deterministic) initial data $v_0$.
We point out that the good set of probability 1 on which almost sure local well-posedness holds
does {\it not} depend on the choice of $v_0 \in H^1(\R^d)$.

\smallskip
In the next two subsections, we state results on   the long time behavior of  solutions
to~\eqref{NLS} and \eqref{NLS2}, using Proposition \ref{PROP:LWP3}.
In particular, we prove almost sure global well-posedness
of the defocusing energy-critical NLS \eqref{NLS} below the energy space
(Theorem \ref{THM:GWP}).
As for NLS \eqref{NLS2} without gauge invariance, 
we  use Proposition \ref{PROP:LWP3}
to construct finite time blowup solutions below the critical regularity
in a probabilistic manner
(Theorem \ref{THM:4}).

\begin{remark}\rm

When $s < 1$, 
the solution map 
\[\Phi: u_0 \in H^s(\R^d) \longmapsto u \in C([-T, T]; H^s(\R^d))\]

\noi
is not continuous for \eqref{NLS} and is not even well defined
for \eqref{NLS2}; see \cite{CCT, II15}.
Once we view $z^\o = S(t) \phi^\o$
as a probabilistically pre-defined data, 
we can  factorize the solution map for \eqref{NLS3} as
\[u_0 = v_0 + \phi^\o \in H^s(\R^d) 
\longmapsto (v_0, z^\o ) \longmapsto v \in C([-T_\o, T_\o]; H^1(\R^d)),\]

\noi
where the first map can be viewed as a universal lift map
and the second map is the solution map $\Psi$ 
to \eqref{NLS4},
which is in fact continuous in $(v_0, z^\o )\in H^1(\R^d) \times S^s([0, T])$, 
where $S^s([0, T]) \subset C([0, T]; H^s(\R^d))$ is the intersection of 
suitable space-time function spaces.  See~\eqref{S0} below for example.
We also point out that under this factorization,
it is clear that the probabilistic component appears only in the first step
while the second step is entirely deterministic.

One can go further and introduce more probabilistically pre-defined objects
in order to improve the regularity threshold.
In the context of the cubic NLS on $\R^3$ \cite{BOP3},  
the first and third authors (with B\'enyi) decomposed
$u$ as $u = z_1^\o + z_3^\o + v$,
where $z_1^\o = S(t) \phi^\o$
and 
$z_3^\o = -i \int_0^t S(t - t') |z_1|^2 z_1(t') dt'$,
thus leading to the following factorization:
\[u_0 = v_0 + \phi^\o \in H^s(\R^3)
\longmapsto (v_0, z_1^\o, z_3^\o ) \longmapsto v \in C([-T_\o, T_\o]; H^1(\R^3)).\]

\noi
The introduction of 
the higher order pre-defined object $z_3$
allowed us to lower the regularity threshold from the previous work  \cite{BOP2}.
For NLS  with non-algebraic nonlinearities such as~\eqref{NLS} and \eqref{NLS2}, 
it is not clear how to introduce a further decomposition at this point.
This is due to the non-smoothness of the nonlinearities.
If one has an algebraic (or analytic) nonlinearity, 
then a Picard iteration yields analytic dependence,
thus enabling us to write a solution as a power series
in terms of initial data, at least in theory.
See \cite{Christ, Onorm}.
On the other hand, if a  nonlinearity is non-smooth, 
then a Picard iteration does not yield analytic dependence,
which makes it hard to find a higher order term.

More recently, the first author (with Tzvetkov and Wang)
proved invariance of the white noise for the (renormalized) cubic fourth order NLS on the circle
\cite{OTW}.
In this work, we introduced an infinite sequence  $\{ z_{2j-1}\}_{j \in \NB}$
of pre-defined objects of order $2j-1$ (depending only on the random initial data) 
and wrote $u = \sum_{j = 1}^ \infty z_{2j-1} + v$, thus considering  the following factorization:
\[u_0^\o \in H^s(\T) 
\longmapsto (z_1^\o, z_3^\o, z_5^\o, \dots ) \longmapsto v \in C(\R; H^s(\T)),\]

\noi
for $s < -\frac 12$,
where $u_0^\o$ is the Gaussian white noise on the circle.
We conclude this remark by pointing out an analogy of this factorization of the ill-posed solution map
to that in the rough path theory \cite{FH}
and more recent studies on stochastic parabolic PDEs
\cite{GP, Hairer}.

\end{remark}

\subsection{Almost sure global well-posedness of the 
defocusing energy-critical   NLS below the energy space}

In this subsection, we consider 
the energy-critical NLS \eqref{NLS}
in the defocusing case (i.e.~with the $+$ sign).
Let us first recall 
the known related result in this direction.
In \cite{BOP2}, the first and third authors (with B\'enyi)
studied the global-in-time behavior of solutions to the defocusing energy-critical cubic NLS \eqref{NLS}
on $\R^4$.
By implementing 
 the probabilistic perturbation theory,
 we  proved 
conditional almost sure global well-posedness
of the defocusing energy-critical  cubic NLS on $\R^4$,
assuming the following energy bound on the residual part $v = u - z$:

\medskip
\noi
{\bf Energy bound:}
Given any $T, \eps > 0$, there exists $R = R(T, \eps)$ and $\O_{T, \eps} \subset \O$ such that
\begin{itemize}
\item[(i)]
$P(\O_{T, \eps}^c) < \eps$, and
\item[(ii)]
If $v = v^\o$ is the solution to \eqref{NLS1a}
for $\o \in \O_{T, \eps}$, then
the following {\it a priori} energy estimate holds:
\begin{equation}
 \|v(t) \|_{L^\infty([0, T]; H^1(\R^d))} \leq R(T, \eps).
\label{HypA}
 \end{equation}

\end{itemize}

\medskip

The main ingredient in this conditional almost sure global well-posedness
result in \cite{BOP2} 
is a perturbation lemma (see Lemma \ref{LEM:perturb} below).
Assuming the energy bound \eqref{HypA} above, 
we iteratively applied the perturbation lemma in the  probabilistic setting
to show that a solution can be extended to a time depending
only on the $H^1$-norm of the residual part~$v$.
Such a perturbative approach  was previously used by
Tao-Vi\c{s}an-Zhang \cite{TVZ} and Killip-Vi\c{s}an with the first  and third authors \cite{KOPV}.
The main novelty in \cite{BOP2} was 
 an application of such a technique in the probabilistic setting,
 allowing us to study the long time behavior of solutions
when there is no invariant measure available for the problem.\footnote{It is worthwhile to mention that
the conditional almost sure global well-posedness in \cite{BOP2}
and Theorem \ref{THM:GWP} below exploit
certain  ``invariance'' property of 
the distribution of the linear solution $S(t) \phi^\o$;
the distribution of
$ S(t) \phi^\o $ on an interval $[t_0, t_0 + \tau*]$
(measured in a suitable space-time norm) depends only on the length $\tau_*$ of the interval.
In \cite{CO}, similar invariance of the distribution of the random linear solution
played an essential role in proving  almost sure global well-posedness.
}

This probabilistic perturbation method can be easily adapted to other critical equations.
In \cite{POC, OP}, 
by establishing the energy bound \eqref{HypA}, 
we implemented the probabilistic perturbation theory 
in the context of the defocusing energy-critical  NLW on $\R^d$, $d = 3, 4, 5$, 
and proved almost sure global well-posedness.

For our problem at hand, 
Proposition \ref{PROP:LWP3} (more precisely Lemma \ref{LEM:LWP4} below)
allows us to 
repeat the argument in~\cite{BOP2}.
Furthermore, we  show that the energy bound \eqref{HypA} 
holds true for $d = 5, 6$ in the defocusing case
and hence we  prove the following almost sure global well-posedness
of the defocusing energy-critical NLS \eqref{NLS}.

\begin{theorem}\label{THM:GWP}
Let $d=5, 6$ and set $s_*= s_*(d)$ by 
\[ \textup{(i)} \ s_* = \frac{63}{68} \ \text{ when $ d= 5$}
\qquad \text{and}\qquad 
\textup{(ii)} \ s_* = \frac{20}{23}
\ \text{ when $ d= 6$.}\]

\noi
Given $\phi \in H^s (\R ^d)$, 
$ s_* <s<1$, 
let $\phi^{\omega}$ be its Wiener randomization defined in \eqref{rand}.
Then, 
 the defocusing energy-critical  NLS \eqref{NLS}  on $\R^d$
 is almost surely globally  well-posed
with respect to the random initial data $\phi^\o$.

More precisely,
there exists a set $\Sigma \subset \O$
with $P(\Sigma) = 1$
such that,
for each $\o \in \Sigma$, there exists a (unique) global-in-time  solution $u$
to \eqref{NLS}
with $u|_{t = 0} = \phi^\o$
in the class
\[ S(t) \phi^\o + C(\R; H^1 (\R^d)) \subset C(\R;H^s(\R^d)).\]

\end{theorem}

In a recent preprint \cite{KMV}, 
Killip-Murphy-Vi\c{s}an 
 studied the defocusing energy-critical cubic NLS with randomized initial data
when $d = 4$.  In particular, under the radial assumption, 
they  proved
almost sure global well-posedness and scattering below the energy space
by implementing a double bootstrap argument
intertwining  the energy and Morawetz estimates.

Our main goal in Theorem \ref{THM:GWP}
is to simply prove almost sure global well-posedness
(without scattering) by establishing the energy bound \eqref{HypA}.
In particular, 
Theorem \ref{THM:GWP}
establishes  the first almost sure global well-posedness result of the defocusing energy-critical NLS \eqref{NLS} below the energy space without the radial assumption.
As mentioned above, the main difficulty in proving Theorem \ref{THM:GWP}
is to establish  the a priori energy bound \eqref{HypA}.
For this purpose, let us recall  the following conservation laws
for \eqref{NLS}:
\begin{align*}
\text{Mass: }&  M(u)(t) = \int_{\R^d} |u(t, x)|^2 dx,\\
\text{Energy: } &  E(u)(t) = \frac 12 \int_{\R^d} |\nb u(t, x)|^2 dx
+\frac {d-2}{2d}  \int_{\R^d} |u(t, x)|^\frac{2d}{d-2} dx.
\end{align*}

\noi
The main task is to control the growth of the energy $E(v)$
for the residual part $v = u - z$
by estimating the time derivative of $E(v)$.
We first point out that while $M(v)$ is not conserved, 
one can easily establish a global-in-time bound on $M(v)$.
See Lemma \ref{LEM:mass}.

By a direct computation with \eqref{NLS1a}, we have
\begin{align}
\dt E(v) 
& =   \Re i \int \big\{|v + z|^\frac{4}{d-2} (v+ z) - |v|^\frac{4}{d-2} v\big\} \Dl \cj vdx\notag\\
& \hphantom{X}- \Re i \int |v+z|^\frac{4}{d-2}(v+z) |v|^\frac{4}{d-2}\cj vdx\notag\\
& =: \1 + \II.
\label{intro1}
\end{align}

\noi
We need to estimate $\dt E(v)$
by $E(v)$ and various norms of the random linear solution $z = S(t) \phi^\o$.
Moreover, we are allowed to use  at most one power of $E(v)$
in order to close a Gronwall-type argument.
Note that the energy $E(v)$ consists of two parts.
On the one hand, while 
the kinetic part controls the derivative of $v$,
its homogeneity (= degree) is low
and hence can not be used to control a nonlinear term of a high degree (in $v$).
On the other hand, 
the potential part has a higher homogeneity
but it can not be used to control any derivative. 
Hence, we need to combine
the kinetic and potential parts of the energy
in an intricate manner.

The main contribution to $\1$ in \eqref{intro1} is given by a term of the form:
\begin{align}
 \int |v|^\frac{4}{d-2} |\nb v \cdot \nb z|dx 
\les \int |\nb v|^2 dx + 
\big\||v|^\frac{4}{d-2}  \nb z\big\|_{L^2_x}^2.
\label{intro2}
\end{align}

\noi
In order to estimate the second term
on the right-hand side, 
we integrate in time and perform multilinear space-time analysis.
More precisely, 
we divide the second term on the right-hand side of \eqref{intro2} 
into a $\theta$-power 
and a $(1-\theta)$-power for some $\theta = \theta (s) \in (0, 1)$
and estimate them in different manners.
As for the $\theta$-power, 
we apply the refinement of the bilinear Strichartz estimate (Lemma \ref{LEM:biStv}),
substitute the Duhamel formula for $v$ (yielding a higher order term in $v$),
and control the resulting contribution (by ignoring the derivative on $v$)
by the potential part of the energy.
We then use the $(1-\theta)$-power to absorb the derivative on $v$
from the $\theta$-power
and control the resulting contribution by the kinetic part of the energy
and the mass.
See Propositions~\ref{PROP:energy} and~\ref{PROP:energy2}.

When $d = 6$, the main contribution to the second term $\II$ in \eqref{intro1}
is given by $\int |v|^3 |z| dx$, 
which can be controlled by (the potential part of) the energy $E(v)$.
On the other hand, 
when $d = 5$, the main contribution to  the second term $\II$ in \eqref{intro1}
is given by $\int |v|^{\frac{11}3} |z| dx$, 
which we can not control by  the energy $E(v)$.
In order to overcome this problem, we use the following modified energy
when $ d= 5$:
\begin{align}
\E(v) = \frac 12 \int |\nb v |^2 dx + \frac 3{10} \int |v+z|^\frac{10}{3} dx .
\label{intro3}
\end{align}

\noi
The use of this modified energy $\E(v)$
eliminates the contribution $\II$ in \eqref{intro1}
at the expense of introducing $\Dl z$ in $\1$.
It turns out, however, the worst term is still given by 
the second term on the right-hand side of \eqref{intro2}
and hence there is no loss in using the modified energy $\E(v)$.

Lastly, we point out  the following.
On the one hand, the regularity for almost sure local well-posedness
in Theorem \ref{THM:LWP1} is worse when $d = 6$.
On the other hand, the regularity   for almost sure global well-posedness
in Theorem \ref{THM:GWP} is worse when $d = 5$:
\[ \frac{20}{23}\approx 0.8696 < 
\frac{63}{68} \approx 0.9265.\]

\noi
This is due to the fact that the main contribution \eqref{intro2}
to the energy estimate
comes with a higher order term in $v$ when $d = 5$.
In fact, when $d = 4$, our argument completely breaks down.
In this case, the left-hand side of \eqref{intro2} becomes
\begin{align*}
 \int |v|^2 |\nb v \cdot \nb z|dx 
\les \int |\nb v|^2 dx + 
\int |v|^4 |\nb z|^2 dx. 
\end{align*}

\noi
Recalling that the potential energy is given by $\frac{1}{4}\int |v|^4 dx$, 
it is easy to see that we can not pass a part of the derivative on 
$z$ to $|v|^4$ in the second term on the right-hand side 
and hence it is not possible to bound it by $\big(E(v)\big)^\alpha$, $\al \leq 1$,
since $z \notin W^{1, p}(\R^4)$ for any $p$,  almost surely.
For this problem, some other space-time control
such as the (interaction) Morawetz estimate\footnote{See \cite{TVZ}
for the interaction Morawetz estimate for NLS with a perturbation.}
is required.\footnote{In a recent preprint \cite{KMV}, 
Killip-Murphy-Vi\c{s}an 
proved
almost sure global well-posedness and scattering below the energy space
for the defocusing energy-critical cubic NLS on $\R^4$ in the radial setting,
where the Morawetz estimate (among other tools available in the radial setting) played an important role.}

\begin{remark}\rm 
In \cite{LM2}, 
L\"uhrmann-Mendelson
used a  modified energy 
with the potential part  given by $\frac{1}{p+1} \int |v + z|^{p+1} dx$
in studying 
 the defocusing energy-subcritical NLW on $\R^3$ ($3 < p< 5$):
\[ \dt^2 u - \Dl u + |u|^{p-1} u = 0\]

\noi
with randomized initial data below the scaling critical regularity.
In particular, 
they adapted the technique from \cite{OP}
and proved almost sure global well-posedness
in $H^s(\R^3) \times H^{s-1}(\R^3)$ for $\frac{p-1}{p+1} < s<1$
by 
 establishing an energy estimate 
for the modified energy.
We point out, however, that 
the use of the modified energy for NLW in \cite{LM2}
is not necessary.  On the contrary,  it provides a worse regularity restriction
than the same argument with the standard energy
for NLW.
In fact, Sun-Xia \cite{SX} independently studied the same problem\footnote{While
the main result in \cite{SX} is stated on the three-dimensional torus $\T^3$,
the same result holds on $\R^3$ by the same proof.}
with the standard energy and proved almost sure global well-posedness
with a better regularity threshold: $\frac{p-3}{p-1} < s < 1$,
which  
interpolates the almost sure global well-posedness
results by Burq-Tzvetkov ($p = 3$) in \cite{BT3} and the first and third authors 
($p = 5$) in \cite{OP}.

While our use of the modified energy $\E(v)$ in \eqref{intro3}
removes the issue with the time derivative of the potential part of the energy
(i.e.~$\II$ in \eqref{intro1}), 
it does not worsen the regularity threshold in the sense that
the worst term is still given by \eqref{intro2}.

\end{remark}

\subsection{Probabilistic construction of  finite time blowup solutions
below the critical regularity}

In this subsection,  we focus on NLS \eqref{NLS2} without gauge invariance.
As compared to the standard NLS \eqref{NLS} with the gauge invariant nonlinearity, 
the equation \eqref{NLS2} is less understood, in particular
due to lack of structures such  as conservation laws.

In recent years, starting with the work by Ikeda-Wakasugi \cite{IW}, 
there has been some development
in the construction of finite time blowup solutions for \eqref{NLS2}, including the case of small initial data.
See also \cite{O1, O2, II2}. 
While there are some variations, 
the criteria for finite time blowup solutions
are very different from those for the standard  NLS \eqref{NLS}
and they are given in terms of a condition on the sign of 
the product of 
the real part (and the imaginary part, respectively) of the coefficient $\ld \in \C\setminus \{0\}$
in \eqref{NLS2}
and the imaginary part (and the real part, respectively)
of (the spatial integral of) an initial condition.
We now recall the result of particular interest due to Ikeda-Inui \cite[Theorem 2.3 and Remark 2.1]{II15}.

Given $v_0 \in H^1(\R^d)$, 
consider NLS  \eqref{NLS2} without gauge invariance 
equipped with an initial condition of the form
$\phi = \al v_0$, $\al \geq 0$.
Moreover, assume that $v_0$ satisfies
\begin{align}
 (\Im \ld)(\Re v_0)(x) & \geq \ind_{|x| \leq 1 } |x|^{-k} \quad \text{for all }x \in \R^d, 
 \label{A1}\\
\text{or } \ \ -  (\Re \ld)(\Im v_0)(x) & \geq \ind_{|x| \leq 1 } |x|^{-k} \quad \text{for all }x \in \R^d
 \label{A2}
 \end{align}

\noi
for some positive $k < \frac d2 -1$.
Then,  there exists $\al_0  = \al_0(d, k , |\ld|)>0$
such that, for any $\al > \al_0$,  
the solution $u= u(\al)$ to \eqref{NLS2} with $u|_{t = 0} = \al v_0$
blows up forward in  finite time.
If we denote $T^*(\al) > 0$ to be the forward maximal time of existence, then
the following estimate holds:
\begin{align}
T^*(\al) \leq C \al^{-\frac{1}{\kappa}}
\label{A3}
\end{align}

\noi
for all $\al > \al_0$, where $\kappa = \frac{d-2}{4} - \frac k2$.
Moreover, we have
\begin{align*}
\lim_{T\to T^*} \| u \|_{L^{q_d}([0, T);W_x^{1, r_d})} = \infty,
\end{align*}

\noi
where $(q_d,r_d)$ is as in \eqref{qr0}.
A similar statement holds for the negative time direction
if we replace \eqref{A1} and \eqref{A2} by 
 $- (\Im \ld)(\Re v_0)(x)  \geq \ind_{|x| \leq 1 } |x|^{-k}$
 and $  (\Re \ld)(\Im v_0)(x)  \geq \ind_{|x| \leq 1 } |x|^{-k}$, 
 respectively.

In the following, we fix $v_0$ satisfying \eqref{A1} or \eqref{A2}
and consider \eqref{NLS2}
with $u|_{t = 0} = \al v_0 + \eps \phi^\o$,
where $\phi^\o$ is the Wiener randomization of some fixed $\phi \in H^s(\R^d)\setminus H^1(\R^d)$, 
$s < 1$, 
$d = 5, 6$.
Namely, we study stability 
of the finite time blowup solution constructed in \cite{II15}
under a rough perturbation in a probabilistic manner.

\begin{theorem}\label{THM:4}
Let $d=5,6$, $1-\frac{1}{d}<s<1$, and $k < \frac d2 - 1$.
Given $\phi \in H^s (\R ^d)$, let $\phi^{\omega}$ be its Wiener randomization defined in \eqref{rand}.
Fix  $v_0 \in H^1(\R^d)$, 
 satisfying \eqref{A1} or \eqref{A2}.
Then, 
for each $R> 0$ and $ \eps > 0$, there exists $\O_{R, \eps} \subset \O$ with 
\[ P(\O_{R, \eps}^c) \leq  C\exp \bigg( - c \frac{R^2}{\eps^2 \|\phi\|_{L^2}^2}\bigg)\]

\noi
and $\al_0 = \al_0(d, k, |\ld|, R, \eps)>0$ such that for each $\o \in \O_{R, \eps}$
and any $\al > \al_0$, 
the solution $u = u^\o$ to \eqref{NLS2} with initial data
\[ u|_{t = 0} = \al v_0 + \eps \phi^\o\]

\noi
blows up forward in finite time
with the  forward maximal time $T^*(\al) $ of existence
satisfying~\eqref{A3}, where the implicit constant depends only on $R> 0$.
Moreover, we have  
\begin{align}
\lim_{T\to T^*} \| u - \eps z^\o\|_{L^{q_d}([0, T);W_x^{1, r_d})} = \infty,
\label{A6}
\end{align}

\noi
where  $z^\o = S(t) \phi^\o$.

\end{theorem}

This result in particular allows us to construct
finite time blowup solutions {\it below} the critical regularity $s_\text{crit} = 1$.
Moreover, it can be viewed as a {\it probabilistic stability} result 
of the finite time blowup solutions in $H^1(\R^d)$ constructed in \cite{II15}
under random and rough perturbations.
Note that $P(\O_{R, \eps})\to 1$ as $\eps \to 0$.

The proof of Theorem \ref{THM:4} is a straightforward combination
of Proposition \ref{PROP:LWP3} and the finite time blowup result in \cite{II15}.
More precisely, we prove Theorem \ref{THM:4} by writing $u = \eps z + v$ 
and considering the equation for the residual term $v$:
\begin{equation}
\begin{cases}
 i\partial _t v + \Delta  v =  \ld |v+\eps z^\o|^\frac{d+2}{d-2}, \\
 v|_{t= 0} =\al v_0,
\end{cases}
\label{NLS5}
\end{equation}

\noi
where $z^\o= S(t) \phi^\o$ as before.
In view of Proposition \ref{PROP:LWP3}, 
the equation \eqref{NLS5} is almost surely locally well-posed
with a blowup alternative \eqref{blowup}.
This allows us to show that the solution $v$ 
is  a weak solution in the sense of Definition \ref{DEF:weak}
and hence to carry out the analysis in \cite{II15} with a small modification
coming from the random perturbation term.
One crucial point to note is that once
we reduce our analysis to the weak formulation in \eqref{weak1},
we only require space-time integrability of the random perturbation $z^\o$
and its differentiability plays no role.
This enables us to prove Theorem \ref{THM:4}.

%

\smallskip

We now give a brief outline of this article.
In Sections \ref{SEC:prob} and \ref{SEC:space}, we recall probabilistic and
deterministic lemmas along with the definitions of the basic  function spaces.
 We  then prove the crucial nonlinear estimates
 in Section \ref{SEC:nonlin},
and  present the proof of the almost sure local well-posedness (Theorems \ref{THM:LWP1}
and \ref{THM:LWP2})
in Section \ref{SEC:thm}.
In Section \ref{SEC:LWP2},
we prove a variant of almost sure local well-posedness (Proposition \ref{PROP:LWP3}).
In Section \ref{SEC:energy}, 
we establish the crucial energy bound \eqref{HypA}
and present the proof of  almost sure global well-posedness
of the defocusing energy-critical NLS \eqref{NLS}
(Theorem \ref{THM:GWP}).
In Section \ref{SEC:7}, 
we use Proposition  \ref{PROP:LWP3}
to construct finite time blowup solutions
below the critical regularity
in a probabilistic manner.

In view of the time reversibility of the equations, 
we only consider positive times in the following.
Moreover, in the local-in-time theory, 
the defocusing/focusing nature of \eqref{NLS}
does not play any role, so we assume that it is defocusing
(with the $+$-sign in \eqref{NLS}).
Similarly, we simply set $\ld = 1$ in~\eqref{NLS2}.

\section{Probabilistic lemmas} \label{SEC:prob}

In this section, we state the probabilistic lemmas used in this paper.
See \cite{BOP1, OP} for their proofs.
The first lemma states that the Wiener randomization  almost surely preserves
the differentiability of a given function.

\begin{lemma} \label{LEM:prob1}
Given $\phi \in H^s (\R ^d)$, let $\phi ^{\omega}$ be its Wiener randomization defined in \eqref{rand}.
Then, there exist $C,c >0$ such that
\[
P\big( \| \phi^\o \|_{H^s} > \lambda \big) \leq C \exp \bigg( - c \frac{\ld^2}{\| \phi \|_{H^s}^2} \bigg)
\]
for all $\lambda >0$.
\end{lemma}

In fact, 
one can also 
show that there is almost surely  no smoothing upon randomization in terms of differentiability
(see, for example, Lemma B.1 in \cite{BT1}).
We, however, do not need 
such a non-smoothing result in the following.

Next, we state the probabilistic Strichartz estimates.
Before doing so,  we first recall the usual Strichartz estimates on $\R^d$ for readers' convenience.
We say that
a pair  $(q, r)$ is  admissible
if $2\leq q, r \leq \infty$, $(q, r, d) \ne (2, \infty, 2)$, and 
\begin{equation}
\frac{2}{q} + \frac{d}{r} = \frac{d}{2}.
\label{Admin}
\end{equation}

\noi
Then, the following Strichartz estimates
are known to hold.
See \cite{Strichartz, Yajima, GV, KeelTao}.

\begin{lemma} \label{LEM:Str}
Let $(q,r)$ be admissible.
Then, we have
\[
\| S(t) \phi \|_{L_t^q L_x^r} \lesssim \| \phi \|_{L^2}.
\]
\end{lemma}

As a corollary, we obtain
\begin{equation}
\| S(t) \phi\|_{L^p_{t, x} }
\lesssim \big\||\nb|^{\frac d2 - \frac{d+2}{p}}\phi\big\|_{L^2}.
\label{Str2}
\end{equation}

\noi
for $p \geq \frac{2(d+2)}{d}$, 
which follows from 
Sobolev's inequality
and  Lemma \ref{LEM:Str}.

The following lemma shows an improvement of the Strichartz estimates
upon the randomization of initial data.
The improvement appears in the form of integrability
and not differentiability.
Note that such a gain of integrability 
is classical in the context of random Fourier series   \cite{PZ}.
The first estimate \eqref{PStr1} follows from 
Minkowski's integral inequality along with Bernstein's inequality.
As for the $L^\infty_T$-estimate \eqref{PStr2}, see \cite{OP}
for the proof (in the context of the wave equation).

\begin{lemma}\label{LEM:prob11}
Given $\phi$ on $\R^d$, 
let $\phi^\o$ be its Wiener randomization defined in \eqref{rand}.
Then,
given finite  $q, r \geq 2$,
there exist $C, c>0$ such that
\begin{align}
P\Big( \|S(t) \phi^\omega\|_{L^q_t L^r_x([0, T)\times \R^d)}> \ld\Big)
\leq C\exp \bigg(-c \frac{\ld^2}{ T^\frac{2}{q}\|\phi\|_{H^s}^{2}}\bigg)
\label{PStr1}
\end{align}
	
\noi
for all  $ T > 0$ and $\lambda>0$
with  \textup{(i)} $s = 0$ if $r < \infty$
and  \textup{(ii)}  $s > 0$ if $r = \infty$.
Moreover, when $q = \infty$, 
given $2 \leq r \leq\infty$, there exist $C, c > 0$ such that 
\begin{align}
P\Big( \|S(t) \phi^\omega\|_{L^\infty_t L^r_x([0, T)\times \R^d)}> \ld\Big)
\leq C(1+T) \exp \bigg(-c \frac{\ld^2}{ \|\phi\|_{H^s}^{2}}\bigg)
\label{PStr2}
\end{align}
	
\noi
for all   $\lambda>0$
with $s>0$.

\end{lemma}

\section{Function spaces and their basic properties} \label{SEC:space}

In this section, we go over the basic definitions
and properties of the functions spaces
used for the Fourier restriction norm method adapted
to the space  of  functions of bounded $p$-variation and its pre-dual,
introduced and developed by 
Tataru, Koch, and their collaborators \cite{KochT, HHK, HTT}.
We refer readers to  Hadac-Herr-Koch \cite{HHK}
and Herr-Tataru-Tzvetkov \cite{HTT} for proofs of the basic properties. See also \cite{BOP2}.

Let $\mathcal{Z}$ be the set of finite partitions $-\infty <t_{0}<t_{1}<\dots <t_{K} \le \infty $ of the real line.
By convention, we set $u(t_K):=0$  if $t_K=\infty$.

\begin{definition} \rm
Let $1\leq p<\infty $. 
We define a $U^p$-atom to be 
a step function $a: \R \to L^2(\R^d)$
of the form 
\begin{equation*}
a=\sum_{k=1}^{K}\phi _{k-1} \chi _{[t_{k-1},t_{k})}, 
\end{equation*}

\noi
where
  $\{t_{k}\}_{k=0}^{K}\in \mathcal{Z}$ and $\{\phi _{k}\}_{k=0}^{K-1}\subset L^2 (\R ^d)$ with 
$\sum_{k=0}^{K-1} \| \phi _{k} \|_{L^{2}}^{p}=1$.
Furthermore, we define the atomic space $U^p = U^p(\R; L^2(\R^d))$
by 
\begin{align*}
U^{p}:= \bigg\{ u: \R \rightarrow L^2 (\R ^d) : u=\sum_{j=1}^{\infty }\lambda _{j}a_{j}\ & \text{for $U^{p}$-atoms $a_{j}$},
\ \{\ld_j\}_{j \in \NB} \in \l^1(\NB; \C)\bigg\}
\end{align*}

\noi
with the norm 
\begin{align*}
\| u \| _{U^{p}}:= \inf \bigg\{ \sum_{j=1}^{\infty }|\lambda _{j}|:u=\sum_{j=1}^{\infty } \lambda _{j}a_{j}
 \text{ for $U^{p}$-atoms } a_{j},  \
 \{\ld_j\}_{j \in \NB} \in \l^1(\NB; \C)\bigg\},
\end{align*}

\noi
where the infimum is taken over all possible representations for $u$.
\end{definition}

\begin{definition}\rm
Let $1\le p<\infty $.

\smallskip

\noi
(i)  We define $V^{p}= V^p(\R; L^2(\R^d))$ to be the space of functions $u:\R\to L^2(\R^d)$ of bounded $p$-variation
with the standard $p$-variation norm
\begin{equation*}
\| u \| _{V^{p}}:=\sup_{\{t_{k}\}_{k=0}^{K}\in \mathcal{Z}}
\bigg( \sum_{k=1}^{K}\| u(t_{k})-u(t_{k-1})\| _{L^{2}}^{p}\bigg) ^{\frac 1p}.
\end{equation*}

\noi
By convention, we impose that  the limits $\lim _{t \to \pm \infty}u(t)$ exist in $L^2 (\R ^d)$.

\smallskip

\noi
(ii) Let $V_{\text{rc}}^{p}$ be the closed subspace of $V^{p}$ 
of all right-continuous functions $u \in V^p$ with  $\lim_{t\rightarrow -\infty }u(t)=0$.
\end{definition}

Recall the following inclusion relation; 
for $1\leq p<q<\infty $, 
\begin{align}
U^p \hookrightarrow V_{\text{rc}}^{p}\hookrightarrow U^{q} \hookrightarrow L^{\infty }(\R ;L^2 (\R ^d)).
\label{incl}
\end{align}

\noi
The space $V^p$ is the classical space of functions of bounded $p$-variation
and the space $U^p$ appears as the pre-dual of $V^{p'}$ with $\frac 1p + \frac 1{p'} = 1$.
Their duality relation and the atomic structure of the $U^p$-space
turned out to be very effective in studying dispersive PDEs in critical settings.

Next, we define the $U^p$- and $V^p$-spaces
adapted to the  Schr\"odinger flow.

\begin{definition} \rm
 Let $1\leq p<\infty $.
We define  $U_{\Delta}^{p}:=S (t) U^{p}$ and ($V_{\Delta}^{p}:=S(t)V^{p}$, respectively)
to be the space of all functions $u:\R \to L^2 (\R ^d)$ such that 
$t \to S ( -t ) u(t)$ is in $U^p$ (and in $V^p$, respectively)
with the norms 
\begin{equation*}
\| u \|_{U_{\Delta}^{p}} :=\| S(-t )u \|_{U^{p}}
\qquad \text{and} \qquad
\| u \|_{V_{\Delta}^p} :=\| S(-t)u \|_{V^{p}}.
\end{equation*}

\noi
 The closed subspace $V_{\text{rc}, \Delta}^{p}$ is defined in an analogous manner.
\end{definition}

Next, we define the dyadically defined versions of $U^p_\Dl$ and $V^p_\Dl$.
We use the convention that capital letters denote dyadic numbers, e.g., $N=2^{n}$ for $n\in \mathbb{N}_{0}:=\mathbb{N}\cup \{ 0\}$.
Fix a nonnegative even function 
$\varphi \in C_{0}^{\infty }((-2,2); [0, 1])$ with $\varphi (r)=1$ for $|r|\leq 1$.
Then, we set $\varphi _{N}(r):=\varphi (r/N) - \varphi (2r /N) $ for $N\geq 2$ and $\varphi _{1}(r):=\varphi (r)$.
Given $N\in 2^{\mathbb{N}_{0}}$, 
let $\P_{N}$ denote the Littlewood-Paley projection operator 
with the Fourier multiplier $\varphi _{N}(|\xi |)$, 
i.e.~$\P_{N}f:=\mathcal{F} ^{-1}[\varphi _{N}(|\xi |)\ft{f}(\xi)]$. 
We also define $\P_{\leq N}:=\sum_{1\leq M\leq N}\P_{M}$ and $\P_{> N}:=\Id-\P_{\leq N}$.

\begin{definition} \label{def:fs} \rm
Let  $s\in \mathbb{R}$.
We define $X^s$ and $Y^s$ as the closures of 
 $C(\R ;H^s(\R ^d)) \cap U_{\Delta}^2 $ 
and  $C(\R ;H^s(\R ^d))\cap V_{\text{rc},\Delta}^2$
 with respect to the norms
\begin{equation*}
\| u \|_{X^s} := \bigg( \sum_{N \in 2^{\NB _0}}N^{2s} \| \P_{N}u \|_{U_{\Delta}^2}^{2}\bigg) ^{\frac 12}
\qquad \text{and}
\qquad 
\| u \|_{Y^s} := \bigg( \sum_{N \in 2^{\NB _0}}N^{2s} \| \P_{N}u \|_{V_{\Delta}^2}^{2}\bigg) ^{\frac 12}, 
\end{equation*}

\noi
respectively.
\end{definition}

The transference principle (\cite[Proposition 2.19]{HHK})
and the interpolation lemma 
\cite[Proposition 2.20]{HHK}
applied on the Strichartz estimates (Lemma \ref{LEM:Str} and \eqref{Str2})
imply the following estimate for the $Y^0$-space.

\begin{lemma} \label{LEM:Stv}
Let $d \ge 1$.  Then, given any  admissible pair  $(q,r)$  with $q>2$
and $p \geq \frac{2(d+2)}{d}$, 
 we have
\begin{align*}
\| u \|_{L_t^q L_x^r} & \lesssim \| u \|_{Y^0}, \\
\|  u \|_{L^p_{t,x}}
&  \les  \big\||\nb|^{\frac d2 - \frac {d+2}p} u\big\|_{Y^0}.
\end{align*}
\end{lemma}

Similarly, the bilinear refinement of the Strichartz estimate \cite{Bou98, CKSTT}
implies the following bilinear estimate.

\begin{lemma} \label{LEM:biStv}
Let $N_1, N_2 \in 2^{\NB _0}$ with $N_1 \le N_2$.
Then, given any  $\eps >0$, we have 
\[
\| \P_{N_1}u_1 \P_{N_2}u_2\|_{L^2_{t,x}} 
\les N_1^{\frac{d-2}{2}} \bigg( \frac{N_1}{N_2} \bigg) ^{\frac{1}{2}-\eps} \| \P_{N_1} u_1 \|_{Y^0} \| \P_{N_2} u_2 \|_{Y^0} 
\]

\noi
for all $u_1, u_2 \in Y^0$.

\end{lemma}

For our analysis, we need to introduce
the local-in-time versions of the spaces defined above.

\begin{definition} \label{def:time_tr}\rm
Let $\mathcal{B}$ be a Banach space consisting of   continuous $H$-valued functions
(in $t \in \R)$
 for some Hilbert space $H$.
We define the corresponding restriction space $\mathcal{B}(I)$ to 
a given time interval $I \subset \R$ as
\begin{equation*}
\mathcal{B}(I):=\big\{ u \in C(I;H): \text{ there exists } v \in \mathcal{B}
\text{ such that } \ v|_{I} = u\big\}.
\end{equation*}

\noi
We 
endow $\mathcal{B}(I)$ with the norm 
\[\|u \|_{\mathcal{B}(I)} :=\inf \big\{ \| v \|_{\mathcal{B}}:
\ v|_{I} = u\big\},\]

\noi
where the infimum is taken over all  possible extensions  $v$ of $u$ onto the real line. 
When $I = [0, T)$, we simply set 
$\mathcal{B}_T : = \mathcal{B}(I) = \mathcal{B}([0, T))$.
\end{definition}

Recall that the space $\mathcal{B}(I)$ is a Banach space.
As a consequence of \eqref{incl}, we have the following inclusion relation;
for any interval $I \subset \R$, 
we have 
\[
X^s(I) \hookrightarrow Y^s(I) \hookrightarrow 
\jb{\nb}^{-s} V^2_{\Delta}(I) \cap C(I;H^s(\R^d)).
\]

We conclude this section by stating the linear estimates.
Given $a \in \R$, we define the integral operator $\I_{a}$
on $ L_{\text{loc}}^{1}( [a,\infty );L^2 (\R ^d) )$ by 
\begin{equation}
\I_{a} [F] (t) :=\int_{a}^{t} S ( t-t') F (t') dt' 
\label{duhamel}
\end{equation}

\noi
for $t\ge a$ and $\I_{a} [F] (t) =0$ otherwise.
When $a = 0$, we simply set $\I = \I_a$.
Given an interval $I = [a, b)$,  
we set the dual norm $N^s(I)$ controlling the nonhomogeneous term on $I$
by 
\[ \| F\|_{N^s(I)} = \big\|\I_a[F]\big\|_{X^s(I)},\]

\noi
we have the following linear estimates.

\begin{lemma} \label{LEM:lin}
Let $s \in \R$ and $T\in ( 0,\infty ]$.
Then,  the following linear estimates hold: 
\begin{align*}
\|   S(t) \phi \|_{X^s_T}
& \leq \|\phi\|_{H^s}, \\
\|  F \|_{N_{T}^{s}}
& \le \sup_{\substack{w \in Y^{-s}_T\\ \| w \|_{Y^{-s}_T=1}}} \left| \int_{0}^{T} \lr{F(t), w(t)}_{L_{x}^{2}} dt \right|
\end{align*}

\noi
 for any $\phi \in H^s(\R^d)$
 and $F \in L^1([0,T);H^{s}(\R ^d))$.
\end{lemma}

The first estimate is immediate from the definition of the space $X^s_T$.
The second estimate basically 
follows from the duality relation between $U^2$ and $V^2$
(\cite[Proposition 2.10, Remark 2.11]{HHK}).
See also Proposition 2.11 in \cite{HTT}.

\section{Nonlinear estimates} \label{SEC:nonlin}

As in Section \ref{SEC:1}, 
let $z(t) = z^{\omega} (t) = S(t) \phi ^{\omega}$ denote the linear solution with the randomized initial data $\phi^\o$
in \eqref{rand}.
If $u$ is a solution to \eqref{NLS}, then the residual term 
 $v = u - z$ satisfies  the perturbed NLS \eqref{NLS1a}.
In this section, we establish relevant nonlinear estimates
in solving the fixed point problem \eqref{NLS1b} for the residual term $v$.

Given $d = 5, 6$, fix an admissible pair: 
\begin{align}
(q_d,r_d) := \bigg( \frac{2d}{d-2}, \frac{2d^2}{d^2-2d+4} \bigg)
= \begin{cases}
\rule[-3mm]{0mm}{0mm}
\big(\frac{10}{3}, \frac{50}{19}\big), & d = 5,\\
\big(3, \frac{18}{7}\big), & d = 6.
\end{cases}
\label{qr1}
\end{align}

\noi
Note that 
\begin{align*}
\frac{d+2}{d-2} q_d' = q_d, 
\end{align*}

\noi
where $q_d' $ denotes the H\"older conjugate of $q_d$.
By Sobolev's inequality, we have 
\begin{align}
W^{1,r_d}(\R^d) \hookrightarrow L^{\rho_d}(\R^d), \qquad \rho_d := \frac{2d^2}{(d-2)^2}
= \begin{cases}
\rule[-3mm]{0mm}{0mm}
 \frac{50}{9}, & d = 5,\\
\frac{9}{2}, & d = 6.
\end{cases}
\label{N2}
\end{align}

\noi
Before we state the main probabilistic nonlinear estimates, 
let us define the set of indices:
\[\mathfrak{S}_{\delta} := \bigg\{ \Big(\frac{q_d}{1-\delta q_d}, r_d\Big), 
\Big(\frac{q_d}{1-\delta q_d}, \, \frac{d+2}{d-2}r_d'\Big), \,  \Big(\frac{q_d}{1-\delta q_d}, \rho_d\Big),
\, \Big(\frac{4}{1 - 4\dl}, 4\Big)  , 
\, \Big(4, \frac{4+2\dl}{\dl}\Big)  \bigg\}
\]

\noi
for small $\dl > 0$.
Given  an interval $I \subset \R$ and $\delta >0$, we define
$S^s(I) = S^s(I; \dl)$ by\footnote{As we see below, we fix $\dl = \dl(d, s) > 0$
and hence we suppress the dependence on $\dl$ for simplicity of the presentation.
A similar comment applies to $E_M(I)$ and $\wt E_M(I)$ defined in \eqref{ER}
and \eqref{ER2}.} 
\begin{equation} \label{S0}
\| u \|_{S^s (I)} := \max \big\{ \| \jb{\nb}^s u \|_{L_t^q L_x^r (I \times \R ^d)} : (q,r) \in \mathfrak{S}_{\delta} \big\} .
\end{equation}

\noi
Furthermore, given $M>0$ and an interval $I$,  define the set $E_M(I) \subset \O$ by 
\begin{align}
E_{M} (I) := \big\{ \omega \in \Omega : \| \phi ^{\omega} \|_{H^s} 
+ \|  S(t) \phi ^{\omega} \|_{S^s (I)} \le M \big\} .
\label{ER}
\end{align}

\noi
When $I = [0, T)$, we simply write $E_{M, T} = E_M([0, T))$.

\begin{proposition} \label{PROP:nonlin1}
Let  $d=5,6$,   $1-\frac{1}{d}<s<1$, and 
\[ \N(u) = |u|^\frac{4}{d-2} u \qquad \text{or}\qquad 
\N(u) = |u|^\frac{d+2}{d-2}.
\]

\noi
Given $\phi \in H^s(\R ^d)$, let $\phi ^{\omega}$ be its Wiener randomization defined in \eqref{rand}
and $z = S(t) \phi^\o$.
Then, there exist  sufficiently small $\delta = \dl (d, s)>0$ 
and $\theta = \theta(d, s)> 0$ such that 
\begin{align}
 \| \mathcal{N} (v+z)\|_{N^1_T}
& \le C_1 \Big\{ \| v \|_{Y^1_T}^{\frac{d+2}{d-2}} + T^{\theta} M^{\frac{d+2}{d-2}} \Big\} , 
\label{nonlin1} \\
\|\mathcal{N} (v_1+z) - \mathcal{N}  & (v_2+z)\|_{N^1_T}   \notag\\
& \le C_2 \Big\{ \| v_1 \|_{Y^1_T}^{\frac{4}{d-2}} + \| v_2 \|_{Y^1_T}^{\frac{4}{d-2}} 
+ T^{\theta} M^{\frac{4}{d-2}} \Big\} \| v_1-v_2 \|_{Y^1_T}, 
 \label{nonlin2}
\end{align}

\noi
for any $T>0$, 
  $v, v_1, v_2 \in Y^1_T$, and $\omega \in E_{M, T}$.

\end{proposition}

Note that we have 
\begin{align}
 \| u \|_{X^1_T} \sim \| u \|_{X^0_T} + \| \nb u \|_{X^0_T}.
 \label{equiv}
\end{align}

\noi
It is crucial that we handle a regular gradient $\nb$ rather than $\jb{\nb}$
for our purpose.
We also point out that once we fix the set $E_{M, T}$,
the nonlinear estimates are entirely {\it deterministic}.

\begin{proof}
{\bf Part 1:}
We first  prove \eqref{nonlin1}.
In view of  \eqref{equiv}, Lemma \ref{LEM:lin} and Definition \ref{def:time_tr}
of the time restriction norm, 
it suffices to show\footnote{Strictly speaking, we need to work with a truncated nonlinearity
as in \cite{BOP2} so that Lemma \ref{LEM:lin} is applicable.
This modification, however, is standard and we omit details.
See \cite{BOP2} for the details.}
\begin{align}
\bigg| \intt _{[0,T] \times \R^{d}}  \mathcal{N}(v+z) \cdot w dx dt \bigg|
\lesssim \| v \|_{Y^1}^{\frac{d+2}{d-2}} + T^{\theta} M^{\frac{d+2}{d-2}},
\label{nonlin3}\\
\bigg| \intt _{[0,T] \times \R^{d}} \nb \mathcal{N}(v+z) \cdot w dx dt \bigg|
\lesssim \| v \|_{Y^1}^{\frac{d+2}{d-2}} + T^{\theta} M^{\frac{d+2}{d-2}},
\label{nonlin3a}
\end{align}

\noi
for all $w \in Y^{0}$ with $\|w \|_{Y^0} = 1$
and any $\omega \in E_{M, T}$,.

Let us first consider \eqref{nonlin3}.
H\"older's inequality and 
the embedding $W^{\frac{4}{d+2}, r_d}(\R^d) \hookrightarrow L^{\frac{d+2}{d-2}r_d'}(\R^d)$ yield
\begin{align}
\text{LHS of \eqref{nonlin3}}
& \les \big\| |v+z|^{\frac{d+2}{d-2}} \big\|_{L_T^{q_d'} L_x^{r_d'}} \| w \|_{L_T^{q_d} L_x^{r_d}}
\les\| v+z \|_{L_T^{q_d} L_x^{\frac{d+2}{d-2} r_d'}}^{\frac{d+2}{d-2}} \notag \\
& \les \| v \|_{Y^1}^{\frac{d+2}{d-2}}+  \|  z \|^{\frac{d+2}{d-2}}_{L_T^{q_d}L_x^{\frac{d+2}{d-2} r_d'}} \notag\\
& \lesssim \| v \|_{Y^1}^{\frac{d+2}{d-2}} + (T^{\delta} M)^{\frac{d+2}{d-2}}
\label{X1}
\end{align}

\noi
for any $\omega \in E_{M, T}$,
where we used
\[
\|  z \|_{L_T^{q_d}L_x^{\frac{d+2}{d-2} r_d'}}
\le T^{\dl }\|  z \|_{L_T^{\frac{q_d}{1-\dl q_d}}L_x^{\frac{d+2}{d-2} r_d'}}
\le T^{\dl }M .\]

Next, we consider \eqref{nonlin3a}.
The contribution from  $\P_{\leq 1}w$ can be estimated 
in an analogous manner to the computation above.
Hence, without loss of generality, we assume   $w=\P_{>1} w$
in the following.

We first prove \eqref{nonlin3a} for  $\N(u) = |u|^\frac{d+2}{d-2}$.
With 
\begin{align}
\nabla (|f|^{\alpha}) = \alpha |f|^{\alpha-2} \Re (f \nabla \overline{f}),
\label{nonlin3b}
\end{align}

\noi
the estimate \eqref{nonlin3a} is reduced to showing 
\begin{equation} \label{nonlin4}
\bigg|\intt _{[0,T] \times \R^{d}}  (\nabla w_1) (v+z)
|v+z|^{\frac{6-d}{d-2}} w dx dt \bigg|\lesssim \| v \|_{Y^1}^{\frac{d+2}{d-2}} + T^{\theta}M^{\frac{d+2}{d-2}}
\end{equation}

\noi
for  $w_1 = \cj v$ or $\cj z$.
A small but important observation is that a derivative
does not fall on the third factor with the absolute value.
In the following, we preform analysis
on the relative sizes of the frequencies of the first two factors.

\medskip
\noi
$\bullet$ {\bf Case 1:}  $w_1=\cj v$.
\quad 
In this case, from  Lemma \ref{LEM:Stv} with \eqref{N2} and \eqref{ER}, we have
\begin{align}
\text{LHS of \eqref{nonlin4} }
& \lesssim  \| \nabla v \|_{L_T^{q_d} L_x^{r_d}} \| v+z \|_{L_T^{q_d} L_x^{\rho_d}} 
\big\| |v+z|^{\frac{6-d}{d-2}}\big\|_{L_T^{\frac{2d}{6-d}} L_x^{\frac{2d^2}{(6-d)(d-2)}}} 
\| w \|_{L_T^{q_d} L_x^{r_d}} \notag \\
& \lesssim \| v \|_{Y^1}  \| v+z  \|_{L_T^{q_d} L_x^{\rho_d}}^{\frac{4}{d-2}}  
\lesssim \| v \|_{Y^1} \big\{ \| v \|_{L_T^{q_d} W_x^{1, r_d}}
+ \| z \|_{L_T^{q_d} L_x^{\rho_d}}\big\}^{\frac{4}{d-2}} \notag \\
& \lesssim \| v \|_{Y^1} \Big\{ \| v \|_{Y^1}^{\frac{4}{d-2}} + (T^{\delta} M)^{\frac{4}{d-2}} \Big\}
\label{X2}
\end{align}

\noi
for any $\omega \in E_{M, T}$.
Then, \eqref{nonlin4} follows from Young's inequality.

\medskip
\noi
$\bullet$ {\bf Case 2:}  $w_1=\cj z$.
\quad 
Using the Littlewood-Paley decomposition, we have
\[
\text{LHS of \eqref{nonlin4} }
\lesssim \sum_{N_1, N_2 \in 2^{\NB_0}} 
\bigg|\intt_{[0,T] \times \R^{d}}N_1 \P_{N_1}\cj  z \, \P_{N_2}(v+z) |v+z|^{\frac{6-d}{d-2}} w  dxdt\bigg|.
\]

\noi
\underline{Subcase 2.a:}
We first consider the contribution from  $N_2 \ges N_1^{\frac{1}{d-1}}$.
Note that we have
\[
\|  z \|_{L_T^{q_d}(W_x^{s, r_d}\cap L_x^{\rho_d})}
\le T^{\dl }\|  z \|_{L_T^{\frac{q_d}{1-\dl q_d}}(W_x^{s, r_d}\cap L_x^{\rho_d})}
\le T^{\dl }M \]

\noi
on $E_{M, T}$.
Then, proceeding as in Case 1 with  Lemma \ref{LEM:Stv},  \eqref{N2}, and \eqref{ER}, we have
\begin{align}
 \text{LHS of }& \eqref{nonlin4} 
 \lesssim \sum_{\substack{N_1, N_2 \in 2^{\NB_0} \\ N_2 \ges N_1^{\frac{1}{d-1}}}} 
N_1 \| \P_{N_1}z \|_{L_T^{q_d} L_x^{\rho_d}} \| \P_{N_2}(v+z) \|_{L_T^{q_d} L_x^{r_d}}  \notag \\
&  \phantom{XXXXXX}
\times \big\| |v+z|^{\frac{6-d}{d-2}} \big\|_{L_T^{\frac{2d}{6-d}} L_x^{\frac{2d^2}{(6-d)(d-2)}}} 
\| w \|_{L_T^{q_d} L_x^{r_d}} \notag \\
& \lesssim \sum_{\substack{N_1, N_2 \in 2^{\NB_0} \\ N_2 \ges N_1^{\frac{1}{d-1}}}} 
N_1^{-s+1} N_2^{-s}
\| \P_{N_1}z \|_{L_T^{q_d} W_x^{s, \rho_d}} 
 \big\{ \|\P_{N_2} v \|_{L_T^{q_d} W_x^{s, r_d}}
+ \| \P_{N_2}z \|_{L_T^{q_d} W_x^{s, r_d}}\big\} \notag \\
&  \phantom{XXXXXX}
\times \big\{ \| v \|_{L_T^{q_d} W_x^{1, r_d}}
+ \| z \|_{L_T^{q_d} L_x^{\rho_d}}\big\}^{\frac{6-d}{d-2}}  \| w \|_{L_T^{q_d} L_x^{r_d}} \notag  \\
& \lesssim \sum_{\substack{N_1, N_2 \in 2^{\NB_0} \\ N_2 \ges N_1^{\frac{1}{d-1}}}} 
N_1^{-s+1} N_2^{-s} T^{\delta} M \Big\{ \| v \|_{Y^1}^{\frac{4}{d-2}} + (T^{\delta} M)^{\frac{4}{d-2}} \Big\} \notag \\
& \les \| v \|_{Y^1}^{\frac{d+2}{d-2}} + (T^{\delta} M)^{\frac{d+2}{d-2}}
\label{X3}
\end{align}

\noi
for any $\omega \in E_{M, T}$,
provided that $s>1-\frac{1}{d}$.

 \smallskip

\noi
\underline{Subcase 2.b:}
Next, we estimate the contribution from  $N_2 \ll N_1^{\frac{1}{d-1}}$.
Noting that  $\big( \frac{4d}{(6-d)(d-2)}, \frac{d^2}{d^2-4d+6} \big)$ is an admissible pair, 
H\"older's inequality and Lemma \ref{LEM:Stv} yield
\begin{align}
\| w \|_{L_T^{\frac{d}{d-3}} L_x^{\frac{d^2}{d^2-4d+6}}} 
\leq T^{\frac d4 - 1}\| w \|_{L_T^{\frac{4d}{(6-d)(d-2)}} L_x^{\frac{d^2}{d^2-4d+6}}} 
\les T^{\frac d4 - 1}\| w \|_{Y^0}.
\label{nonlin4a}
\end{align}

\noi
Then,  by applying Lemma  \ref{LEM:biStv} 
with Lemmas  \ref{LEM:Stv} and \ref{LEM:lin}
and \eqref{ER}, we obtain 
\begin{align}
 \text{LHS of \eqref{nonlin4}} 
& \lesssim \sum_{\substack{N_1, N_2 \in 2^{\NB_0} \\ N_2 \ll N_1^{\frac{1}{d-1}}}} 
N_1 \| \P_{N_1}z \P_{N_2}(v+z) \|_{L^2_{T, x}} 
\notag \\
&  \phantom{XXXXXX}
\times \big\| |v+z|^{\frac{6-d}{d-2}} \big\|_{L_T^{\frac{2d}{6-d}} L_x^{\frac{2d^2}{(6-d)(d-2)}}} 
\| w \|_{L_T^{\frac{d}{d-3}} L_x^{\frac{d^2}{d^2-4d+6}}} \notag \\
& \lesssim T^{\frac{d}{4}-1} \sum_{\substack{N_1, N_2 \in 2^{\NB_0} \\ N_2 \ll N_1^{\frac{1}{d-1}}}} 
N_1 N_2^{\frac{d}{2}-1} \bigg( \frac{N_2}{N_1} \bigg)^{\frac{1}{2}-\eps} 
\| \P_{N_1} z \|_{Y^0_T} \| \P_{N_2} (v+z) \|_{Y^0_T} 
\notag \\
&  \phantom{XXXXXX}
\times 
\big\{ \| v \|_{L_T^{q_d} W_x^{1, r_d}}
+ \| z \|_{L_T^{q_d} L_x^{\rho_d}}\big\}^{\frac{6-d}{d-2}}  \| w \|_{Y^0} \notag \\
& \lesssim T^{\frac{d}{4}-1} \sum_{\substack{N_1, N_2 \in 2^{\NB_0} \\ N_2 \ll N_1^{\frac{1}{d-1}}}} 
N_1^{-s+\frac{1}{2}+\eps} N_2^{-s+\frac{d-1}{2}-\eps} M \notag \\
&  \phantom{XXXXXX}
\times 
( \| v \|_{Y^s} + M ) 
\Big\{ \| v \|_{Y^1}^{\frac{6-d}{d-2}} + (T^{\delta}M)^{\frac{6-d}{d-2}} \Big\} \notag \\
& \lesssim T^{\theta'} M \Big\{ \| v \|_{Y^1}^{\frac{4}{d-2}} + M^{\frac{4}{d-2}} \Big\}\notag\\
& \lesssim \| v \|_{Y^1}^{\frac{d+2}{d-2}} + T^{\theta} M^{\frac{d+2}{d-2}}
\label{X4}
\end{align}

\noi
for any $\omega \in E_{M, T}$,
provided that $s>1-\frac{1}{d}$.
This proves  \eqref{nonlin1}
for  $\N(u) = |u|^\frac{d+2}{d-2}$.

We now  prove \eqref{nonlin3a} for  $\N(u) = |u|^\frac{4}{d-2} u $.
In this case, we have\footnote{Here, we assumed that $\dd \{ x \in \R^d: f(x) = 0\}$
has measure 0. This assumption can be verified for smooth truncated 
$\P_{\leq N} z$
and smooth $v_N$. 
Then, we can establish the desired estimates 
for smooth $\P_{\leq N} z$ and $v_N$ and take a limit as $N\to \infty$.}
 \begin{align}
 \nabla (|f|^{\alpha-1}f) = 
 (\al -1) |f|^{\alpha-2} \tfrac{f}{|f|}  \Re (f \nabla \overline{f})
 +  |f|^{\alpha-1} \nb f.
 \label{nonlin5}
 \end{align}

\noi
Noting that  
$\big||f|^{\alpha-3} f\big|  = |f|^{\al - 2}$, 
we can estimate  the first term in \eqref{nonlin5}
using \eqref{nonlin4}.
It remains to estimate the contribution from the second term in \eqref{nonlin5}. 
Namely, we prove
\begin{equation} \label{nonlin6}
\bigg|\intt _{[0,T] \times \R^{d}}  
(\nabla w_1) |v+z|^{\frac{4}{d-2}} w  
dx dt\bigg| \lesssim \| v \|_{Y^1}^{\frac{d+2}{d-2}} + T^{\theta}M^{\frac{d+2}{d-2}}
\end{equation}

\noi
for  $w_1 = v$ or $z$.
When $w_1 = v$, \eqref{nonlin6} follows from Case 1 above.
Hence, we assume that $w_1 = z$ in the following.
By writing
$(\nabla z) |v+z|^{\frac{4}{d-2}} 
=(\nabla z) |v+z|\cdot  |v+z|^{\frac{6-d}{d-2}} $, 
it follows from 
 Lemma \ref{LEM:Stv} and  \eqref{nonlin4a} with \eqref{ER} that 
\begin{align}
 \text{LHS of \eqref{nonlin6}} 
& \les
 \| (\nb z) (v+z) \|_{L^2_{T, x}} 
 \big\| |v+z|^{\frac{6-d}{d-2}} \big\|_{L_T^{\frac{2d}{6-d}} L_x^{\frac{2d^2}{(6-d)(d-2)}}} 
\| w \|_{L_T^{\frac{d}{d-3}} L_x^{\frac{d^2}{d^2-4d+6}}}\notag\\
& \lesssim \sum_{N_1, N_2 \in 2^{\NB_0}}
N_1 \| \P_{N_1}z \P_{N_2}(v+z) \|_{L^2_{T, x}}  \notag\\
&  \phantom{XXXXXX}
\times
 \big\| |v+z|^{\frac{6-d}{d-2}} \big\|_{L_T^{\frac{2d}{6-d}} L_x^{\frac{2d^2}{(6-d)(d-2)}}} 
\| w \|_{L_T^{\frac{d}{d-3}} L_x^{\frac{d^2}{d^2-4d+6}}}\notag \\
& \lesssim T^{\frac{d}{4}-1} 
\Big\{ \| v \|_{Y^1}^{\frac{6-d}{d-2}} + (T^{\delta}M)^{\frac{6-d}{d-2}} \Big\} 
 \sum_{N_1, N_2 \in 2^{\NB_0}}
N_1 \| \P_{N_1}z \P_{N_2}(v+z) \|_{L^2_{T, x}} 
\label{nonlin7}
\end{align}

\noi
for any $\omega \in E_{M, T}$.
When $N_2 \ll N_1^{\frac{1}{d-1}}$, 
we can apply Lemma \ref{LEM:biStv} as in  Subcase 2.b 
and establish  \eqref{nonlin6}.

Let us consider the remaining case
$N_2 \ges N_1^{\frac{1}{d-1}}$.
As in Subcase 2.a, we have
\begin{align*}
\sum_{\substack{N_1, N_2 \in 2^{\NB_0} \\ N_2 \ges N_1^{\frac{1}{d-1}}}} 
N_1 \| \P_{N_1}z \P_{N_2}z \|_{L^2_{T, x}} 
& \les \sum_{\substack{N_1, N_2 \in 2^{\NB_0} \\ N_2 \ges N_1^{\frac{1}{d-1}}}} 
N_1^{-s+1} N_2^{-s}
\| \P_{N_1}z \|_{L^4_{T} W^{s, 4}_x} \|  \P_{N_2}z \|_{L^4_{T} W^{s, 4}_x} \notag \\
& \les (T^{\delta}M)^2
\end{align*}

\noi
for any $\omega \in E_{M, T}$,
provided that $s>1-\frac{1}{d}$.
Similarly,  it follows from Sobolev's inequality
(with sufficiently small $\dl > 0$ such that $\frac{1-s}{d} \geq \frac 12 - \frac 1{2+\dl}$)
and \eqref{ER} that 
\begin{align*}
\sum_{\substack{N_1, N_2 \in 2^{\NB_0} \\ N_2 \ges N_1^{\frac{1}{d-1}}}} 
N_1 \| \P_{N_1}z \P_{N_2}v \|_{L^2_{T, x}} 
& \les \sum_{\substack{N_1, N_2 \in 2^{\NB_0} \\ N_2 \ges N_1^{\frac{1}{d-1}}}} 
N_1^{-s+1} N_2^{-s}
\|  \P_{N_1}z \|_{L^4_T W^{s, \frac{4+2\dl}{\dl}}_x} 
\|\P_{N_2}v \|_{L^4_T W^{s, ^{2+\dl}}_x} \notag \\
& \les T^\frac{1}{4} \| v\|_{Y^1}M
\end{align*}
 
\noi
for any $\omega \in E_{M, T}$,
provided that $s>1-\frac{1}{d}$.
This proves  \eqref{nonlin1}
for  $\N(u) = |u|^\frac{4}{d-2}u$.

\medskip

\noi
{\bf Part 2:}
Next, we prove the difference estimates \eqref{nonlin2}.
Our main goal is to prove 
\begin{align}
\bigg|\intt _{[0,T] \times \R^{d}}  
 \big\{ \N(v_1 + z) & - \N(v_2 + z)\big\} w dx dt \bigg| \notag\\
& \les  \Big\{ \| v_1 \|_{Y^1}^{\frac{4}{d-2}} + \| v_2 \|_{Y^1}^{\frac{4}{d-2}} 
+ T^{\theta} M^{\frac{4}{d-2}} \Big\} \| v_1-v_2 \|_{Y^1}, 
\label{diff-1}
\end{align}

\noi
and 
\begin{align}
\bigg|\intt _{[0,T] \times \R^{d}}  
 \big\{\nabla \N(v_1 + z) & - \nb\N(v_2 + z)\big\} w dx dt \bigg| \notag\\
& \les  \Big\{ \| v_1 \|_{Y^1}^{\frac{4}{d-2}} + \| v_2 \|_{Y^1}^{\frac{4}{d-2}} 
+ T^{\theta} M^{\frac{4}{d-2}} \Big\} \| v_1-v_2 \|_{Y^1}
\label{diff0}
\end{align}

\noi
for all $w \in Y^{0}$ with $\|w \|_{Y^0} = 1$.
In the following, we only consider  \eqref{diff0}
and discuss how to apply the computations in Part 1.
The first difference estimate \eqref{diff-1}
follows in a similar, but simpler manner.

\medskip

\noi
$\bullet$ {\bf Case 3:}  $\N(u) = |u|^\frac{d+2}{d-2}$.
\quad Let $F(\zeta) = F(\zeta, \cj \zeta) = |\zeta|^\frac{6-d}{d-2} \zeta$.
Then, we have
\begin{align}
 \dd_\zeta F = \tfrac{2+d}{2d-4} |\zeta|^\frac{6-d}{d-2}
\qquad \text{and} \qquad
\dd_{\cj \zeta} F = \tfrac{6-d}{2d-4} |\zeta|^\frac{6-d}{d-2} \tfrac{\zeta^2}{|\zeta|^2}.
\label{diff0a}
\end{align}

\noi
By Fundamental Theorem of Calculus, we have
\begin{align}
F(v_1+z) - F(v_2 + z) 
= \int_0^1 \dd_\zeta F(v_2 & + z + \theta(v_1 - v_2))(v_1- v_2)\notag\\
& + \dd_{\cj \zeta}F(v_2 + z + \theta(v_1 - v_2))(\cj {v_1- v_2}) d\theta.
\label{diff1}
\end{align}

\noi
Then, from \eqref{nonlin3b} and \eqref{diff1}, we have
\begin{align}
\nabla (|v_1 + z|^\frac{d+2}{d-2}) 
& - \nabla (|v_2 + z|^\frac{d+2}{d-2}) \notag\\
& = \tfrac{d+2}{d-2} \Re\big\{ F(v_1 + z) \nabla (\cj{v_1 + z})
-  F(v_2 + z) \nabla (\cj{v_2 + z})\big\}\notag\\
& = \tfrac{d+2}{d-2} \Re\bigg\{ F(v_1 + z) \nabla (\cj{v_1 - v_2}) \notag\\
& \hphantom{X}
+  \int_0^1 \dd_\zeta F(v_2  + z + \theta(v_1 - v_2))(v_1- v_2)d\theta
\cdot  \nabla (\cj{v_2 + z})\notag\\
& \hphantom{X}
+ \int_0^1  \dd_{\cj \zeta}F(v_2 + z + \theta(v_1 - v_2))(\cj{ v_1-  v_2}) d\theta
\cdot  \nabla (\cj{v_2 + z})  \bigg\}.
\label{diff2}
\end{align}

The contribution to \eqref{diff0} from the first term on the right-hand side of \eqref{diff2}
can be estimated as in \eqref{nonlin4}.
As for the second term on the right-hand side of~\eqref{diff2}, 
the estimate~\eqref{diff0} is reduced to 
\begin{align*}
 \int_0^1 \bigg|\intt _{[0,T] \times \R^{d}}  
(\nb w_1) (v_1 - v_2)
& \cdot \big|  v_2  +  z + \theta(v_1 - v_2)\big|^\frac{6-d}{d-2}
  w dx dt \bigg|  d\theta \notag\\
& \les  \Big\{ \| v_1 \|_{Y^1}^{\frac{4}{d-2}} + \| v_2 \|_{Y^1}^{\frac{4}{d-2}} 
+ T^{\theta} M^{\frac{4}{d-2}} \Big\} \| v_1-v_2 \|_{Y^1}
\end{align*}

\noi
for  $w_1 = \cj v_2$ or $\cj z$,
which once again follows from \eqref{nonlin4} in Part 1.
In view of \eqref{diff0a}, we have 
 $|\dd_{\cj \zeta} F| \sim |\zeta|^\frac{6-d}{d-2}$.
 Hence, 
the third  term on the right-hand side of \eqref{diff2}
can be estimated in a similar manner.

\medskip

\noi
$\bullet$ {\bf Case 4:}  $\N(u) = |u|^\frac{4}{d-2}u $.
\quad 
In view of \eqref{nonlin5}, there are two contributions to 
\[\nabla \N(v_1 + z)  - \nb \N(v_2 + z).\]

\noi
Let $G(\zeta) = G(\zeta, \cj \zeta) = |\zeta|^\frac{8-2d}{d-2} \zeta^2$.
Then, we have
\begin{align}
 \dd_\zeta G = \tfrac{d}{d-2} |\zeta|^\frac{6-d}{d-2} \tfrac{\zeta}{|\zeta|}
\qquad \text{and} \qquad
\dd_{\cj \zeta} G = \tfrac{4-d}{d-2} |\zeta|^\frac{6-d}{d-2} \tfrac{\zeta^3}{|\zeta|^3}.
\label{diff4}
\end{align}

\noi
Next, let $H(z) = H(\zeta, \cj \zeta) = |\zeta|^\frac{4}{d-2} $.
Then, we have
\begin{align}
 \dd_\zeta H = \tfrac{2}{d-2} |\zeta|^\frac{6-d}{d-2} \tfrac{\cj \zeta}{|\zeta|}
\qquad \text{and} \qquad
\dd_{\cj \zeta} H = \tfrac{2}{d-2} |\zeta|^\frac{6-d}{d-2} \tfrac{\zeta}{|\zeta|}.
\label{diff5}
\end{align}

\noi
Then, from \eqref{nonlin5},  \eqref{diff4}, 
and  \eqref{diff5}, 
we have
\begin{align*}
\nabla \N(v_1 + z) &  - \nb \N(v_2 + z)\notag\\
& = \tfrac{4}{d-2}\Re\big\{G(v_1 + z)\nb (\cj{v_1 + z})
- G(v_2 + z)\nb (\cj{v_2 + z})\big\}\notag\\
& \hphantom{XX}
+ H(v_1 + z)\nb (v_1 + z)
- H(v_2 + z)\nb (v_2 + z).
\end{align*}

\noi
Noting that 
\[ |\dd_{ \zeta} G| \sim |\dd_{\cj \zeta} G| \sim |\dd_{ \zeta} H| \sim |\dd_{\cj \zeta} H|
\sim  |\zeta|^\frac{6-d}{d-2},\]

\noi
we can use \eqref{diff1} with $G$ and $H$ replacing $F$
and repeat the computation in Part 1 to establish \eqref{diff0}.
This completes the proof of Proposition \ref{PROP:nonlin1}.
\end{proof}

\section{Proof of Theorems \ref{THM:LWP1} and \ref{THM:LWP2}} \label{SEC:thm}

We present the  proof of Theorems \ref{THM:LWP1} and \ref{THM:LWP2}.
Namely, we solve the following fixed point problem:
\[ v = -i  \I [\mathcal{N} (v+z)],\]

\noi
where 
\[ \N(u) = |u|^\frac{4}{d-2} u \qquad \text{or}\qquad 
\N(u) = |u|^\frac{d+2}{d-2}.
\]

Let $\eta>0$ be sufficiently small  such that
\begin{equation*}
2 C_1 \eta^{\frac{4}{d-2}} \le 1 
\qquad \text{and} \qquad 3 C_2 \eta ^{\frac{4}{d-2}} \le \tfrac{1}{2},
\end{equation*}

\noi
where $C_1$ and $C_2$ are the constants  in \eqref{nonlin1} and \eqref{nonlin2}.
Given  $M>0$, we set 
\begin{align}
T := \min \Big\{\big( \tfrac{\eta}{M} \big)^{\frac{d+2}{d-2}}, 
\big( \tfrac{\eta}{M} \big)^{\frac{4}{d-2}}\Big\}^\frac{1}{\theta}.
\label{LWP1}
\end{align}

\noi
Then, it follows from  Proposition \ref{PROP:nonlin1} with $X^1_T \hookrightarrow Y^1_T$
that for each $\o \in E_{M, T}$,  
the mapping $v \mapsto -i \I [\mathcal{N} (v+z)]$ is a contraction on the ball $B_{\eta}\subset X^1_T$ defined by
\[
B_{\eta} := \{ v \in X^1_T : \| v \|_{X^1_T} \le \eta \}.
\]

\noi
Moreover, it follows from 
Lemmas \ref{LEM:prob1} and \ref{LEM:prob11} with \eqref{LWP1} imply the following tail estimate:
\begin{align*}
P(\Omega \setminus E_{M, T}) 
& \le C \exp \bigg( -c \frac{M^2}{\| \phi \|_{H^s}^2} \bigg) 
+ C \exp \bigg( -c \frac{M^2}{T^\g \| \phi \|_{H^s}^2} \bigg)\\ 
& \leq C \exp \bigg( - \frac{c}{T^\g \| \phi \|_{H^s}^2} \bigg) 
\end{align*}

\noi
for some $\gamma > 0$.
This proves almost sure local  well-posedness
of \eqref{NLS} and \eqref{NLS2}.

\section{A variant of almost sure local well-posedness}\label{SEC:LWP2}

In this section, we briefly discuss the proof of Proposition \ref{PROP:LWP3}.
In particular, we consider the perturbed NLS \eqref{NLS4}
with a non-zero initial condition $v_0$. 
This will be useful in proving Theorems \ref{THM:GWP} and 
\ref{THM:4}.
As in \cite{BOP2}, 
we consider the following Cauchy problem for NLS with a perturbation:
\begin{equation}
\begin{cases}
 i\partial _t v + \Delta  v =  \mathcal{N} (v+f), \\
 v|_{t= 0} =v_0 \in H^1(\R^d) , 
\end{cases}
\label{NLS5a}
\end{equation}

\noi
where $f$ is a given deterministic function, satisfying certain regularity conditions.
This allows us to separate the probabilistic and deterministic components
of the argument in a clear manner.

First, note that, since our initial condition is not 0, 
the $Y^1_T$-norm of the solution $v$ does not tend to 0 even when $T \to 0$.
Hence, we need to use an auxiliary norm
that tends to 0 as $T \to 0$.
As a corollary to (the proof of) Proposition \ref{PROP:nonlin1}, we obtain the following
nonlinear estimates, which are stated for a general time interval $I \subset \R$.
Note that 
all the terms on the right-hand side in the first estimate \eqref{Y1} have
(i) two factors of the $L^{q_d}_t(I;  W_x^{1, r_d})$-norm of $v$ (which is weaker than the $X^1(I)$-norm)
or (ii) a factor of $|I|^\theta$,  which can be made small by shrinking the interval $I$.

In the following, 
let  $(q_d, r_d)$ be the admissible pair defined in \eqref{qr1}.
Given $\dl > 0$, $M>0$,  and an interval $I$, define  $\wt E_M(I)$ by 
\begin{align}
\wt E_{M}(I) := \big\{ f \in Y^s(I) \cap S^s(I):    \| f \|_{Y^s}(I) 
+ \| f \|_{S^s (I )} \le M \big\} ,
\label{ER2}
\end{align}

\noi
where $S^s(I) = S^s(I; \dl)$ is as in \eqref{S0}.
\noi
When $I = [0, T)$, we simply write $\wt E_{M, T} = \wt E_M([0, T))$.

\begin{corollary} \label{COR:nonlin2}
Let  $d=5,6$,   $1-\frac{1}{d}<s<1$, and 
\[ \N(u) = |u|^\frac{4}{d-2} u \qquad \text{or}\qquad 
\N(u) = |u|^\frac{d+2}{d-2}.
\]

\noi
Then, there exist  sufficiently small $\delta = \dl (d, s)>0$ 
and $\theta = \theta(d, s)> 0$ such that 
\begin{align}
&  \|  [\mathcal{N} (v+f)\|_{N^1(I)}
 \les  
\|v\|_{L^{q_d}_t(I;  W_x^{1, r_d})}^{\frac{d+2}{d-2}}
 + |I|^{\theta} M^{\frac{d+2}{d-2}} 
+ |I|^{\theta} M \|v\|_{L^{q_d}_t(I;  W_x^{1, r_d})}^{\frac{6-d}{d-2}} \| v \|_{Y^1(I)}, 
\label{Y1} \\
& \| \mathcal{N} (v_1+f) - \mathcal{N}   (v_2+f)\|_{N^1(I)}   \notag\\
& \hphantom{XXXXX}
\les  \Big\{ \| v_1 \|_{L^{q_d}_t(I; W_x^{1, r_d})}^{\frac{4}{d-2}} + \| v_2 \|_{L^{q_d}_t(I; W_x^{1, r_d})}^{\frac{4}{d-2}}
+ |I|^{\theta} M^{\frac{4}{d-2}} \Big\} \| v_1-v_2 \|_{Y^1(I)},
 \label{Y2}
\end{align}

\noi
for any interval $I \subset \R$, 
 $v, v_1, v_2 \in Y^1(I)$, and $f \in \wt E_M(I)$.

\end{corollary}

\begin{proof}
This corollary follows from the proof of Proposition \ref{PROP:nonlin1}
simply by not applying the Strichartz estimates (Lemma \ref{LEM:Stv}).
In particular, 
a small modification to 
\eqref{X1}, \eqref{X2}, and~\eqref{X3} 
yields \eqref{Y1} for the corresponding cases, where the left-hand side is controlled by 
the first two terms on the right-hand side of \eqref{Y1}.
In \eqref{X4} and \eqref{nonlin7}, 
the subcritical nature of the perturbation $f$ 
allows us to gain a small power of $|I|$ through \eqref{nonlin4a}.
Hence,  we obtain~\eqref{Y1},
 where the left-hand side is controlled by 
the last two terms on the right-hand side of \eqref{Y1}.
The difference estimate \eqref{Y2} also follows from a similar modification.
\end{proof}

By following the proof of Proposition 6.3 in \cite{BOP2}, 
we obtain the following almost sure local well-posedness
of the perturbed NLS \eqref{NLS5a} with non-zero initial data.
Proposition  \ref{PROP:LWP3} in Section \ref{SEC:1} then follows from 
this lemma 
with 
Lemmas \ref{LEM:prob1} and \ref{LEM:prob11}
by setting $f = z^\o = S(t) \phi^\o$.

\begin{lemma}\label{LEM:LWP4}
Assume the hypotheses of Corollary \ref{COR:nonlin2}. 
Given $M > 0$, let $\wt E_M(\cdot) $ be as in~\eqref{ER2}
and let $\theta > 0$ be as in Corollary \ref{COR:nonlin2}.
Then, 
 there exists  small $\eta_0 = \eta_0(\|v_0\|_{H^1}, M)>0$ such that
if
\begin{align*}
 \|S(t-t_0)v_0\|_{L^{q_d}_t(I;  W_x^{1, r_d})}\leq \eta
\qquad \text{and}
\qquad
|I| \leq  \eta^\frac{2}{\theta}
\end{align*}

\noi
for some $\eta \leq \eta_0$ and
some time interval  $I = [t_0, t_1] \subset \R$,
then 
for any $f \in \wt E_M(I)$, 
there exists a unique solution $v \in X^1(I)\cap C(I; H^1(\R^d))$
to \eqref{NLS4}
with $v|_{t = t_0} = v_0$, satisfying 
\begin{align*}
\|v \|_{L^{q_d}_t(I;W_x^{1, r_d})} & \leq 2\eta,\\
\| v - S(t - t_0) v_0 \|_{X^1(I)} & \les \eta.
\end{align*}

\end{lemma}
	
\begin{proof}
As mentioned above, one can  prove Lemma \ref{LEM:LWP4} 
by following the proof of Proposition~6.3 in \cite{BOP2}.
More precisely,  
by applying 
Corollary  \ref{COR:nonlin2}
and choosing
\begin{align*}
\eta_0  \ll  \wt R^{-\frac{d+2}{d-2}} 
\end{align*}

\noi
with  $\wt R :=\max(\| v_0\|_{H^1} , M)$,
a straightforward computation shows that 
 the map $\G$ defined by
\begin{equation*}
\G v(t) := S(t-t_0) v_0 -i  \int_{t_0}^t S(t -t') \N(v+f)(t') dt'
\end{equation*}

\noi is a contraction on
\[ B_{R, M,  \eta} = \big\{ v \in X^1(I)\cap C(I; H^1):
 \, \|v\|_{X^1(I)}\leq 2\wt R,
\ \|v\|_{L^{q_d}_t(I;W_x^{1, r_d})} \leq 2\eta\big\},\]

\noi
provided that  $f \in \wt E_M(I)$.
\end{proof}

\medskip

Lastly, note that Lemma \ref{LEM:LWP4}
yields the following blowup alternative.
Suppose that there exists $M(t)$
such that  $f \in \wt E_{M(t)}([0, t))$ for each $t > 0$.
Then, 
given $v_0 \in H^1(\R^d)$, 
let $v$ be the solution to the perturbed NLS \eqref{NLS5a}
with $v|_{t = 0} = v_0$ on a forward maximal time interval
$[0, T^*)$ of existence.
Then, either $T^* = \infty$
or 
\begin{align}
\lim_{T\to T^*} \| v \|_{L^{q_d}_t([0, T);W_x^{1, r_d})} = \infty.
\label{Y3}
\end{align}

\noi
In view of Lemma \ref{LEM:LWP4}, 
this blowup alternative 
follows from a standard argument as in 
 \cite{Cazenave}.
In fact, suppose $T^* < \infty$ and 
\begin{align*}
A^*  := \lim_{T\to T^*} \| v \|_{L^{q_d}_t([0, T);W_x^{1, r_d})} < \infty.
\end{align*}

\noi
Then, we will derive a contradiction in the following.

Without loss of generality, assume that $M(t)$ is non-decreasing
and set 
\begin{align}
M^* := \sup_{t \in [0, T^*+1]} M(t) < \infty.
\label{Y5}
\end{align}

\noi
Partition the interval $[0, T^*]$ as
\[ [0, T^*] = \bigcup_{j = 0}^{J} I_j \cap [0, T^*]\]

\noi
where
$ I_j = [t_j , t_{j+1}]$ with  $ t_0 = 0$ and  $t_{J+1} = T^*$.
From \eqref{Y1} in Corollary \ref{COR:nonlin2} with Lemma~\ref{LEM:lin}, we have
\begin{align*}
\| v \|_{X^1(I_j)} 
& \leq \|v(t_j)\|_{H^1} + \| \mathcal{N} (v+z)\|_{N^1(I_j)}\notag\\
&  \le  \|v(t_j)\|_{H^1}
+  C(T^*, A^*, M^*) 
+ |I_j|^{\theta} M^* (A^*)^{\frac{6-d}{d-2}} \| v \|_{X^1(I_j)}.
\end{align*}

\noi
Hence by  imposing 
that the lengths of the subintervals $I_j$
are sufficiently small, depending only on $A^*$ and $M^*$, 
we obtain 
\begin{align}
\sup_{t \in I_j} \| v(t)\|_{H^1} \les \| v \|_{X^1(I_j)} 
&  \les  \|v(t_j)\|_{H^1}
+  C(T^*, A^*, M^*) ,
\label{Y6}
\end{align}

\noi
where the implicit constants are independent of $j = 0, 1, \dots, J$.
By iteratively applying the estimate \eqref{Y6}, we obtain
\begin{align}
R^* : = \sup_{t \in [0, T^*]} \|v(t) \|_{H^1} \leq   C(T^*, A^*, M^*)  < \infty.
\label{Y7}
\end{align}

\noi
Then, combining \eqref{Y6} and \eqref{Y7}, we obtain
\begin{align}
 \| v \|_{X^1(I_j)} 
&  \leq   C(T^*, A^*, M^*)  < \infty
\label{Y8}
\end{align}

\noi
uniformly in $j = 0, 1, \dots, J$.

Given $\wt \eta > 0$ (to be chosen later),  
we refine the partition and assume that 
\begin{align}
 \| v \|_{L^{q_d}_t(I_j;W_x^{1, r_d})} < \wt \eta.
\label{Y9}
\end{align}

\noi
Fix  $\eta_0 = \eta_0(R^*, M^*)>0$, where $\eta_0$ is as in   Lemma \ref{LEM:LWP4}
and $R^*$ and $M^*$ are as in \eqref{Y7} and~\eqref{Y5}.
Then, by taking the $L^{q_d}_t(I_j;W_x^{1, r_d})$-norm of the Duhamel formulation:
\[S(t-t_j) v(t_j) = v(t) +i \int_{t_j}^t S(t - t')\N(v+f) dt', \]

\noi
applying Corollary \ref{COR:nonlin2} with 
\eqref{Y8} and 
the smallness condition \eqref{Y9}, 
and taking $\wt \eta  = \wt\eta(\eta_0) = \wt \eta (R^*, M^*)> 0$ 
and $|I_j| = |I_j|(T^*, A^*,  M^*, \eta_0)$ sufficiently small, 
we have
\begin{align*}
\|S(t-t_j) v(t_j) \|_{L^{q_d}_t(I_j;W_x^{1, r_d})}
&\leq \wt \eta  + C 
\wt \eta^\frac{d+2}{d-2} + C(T^*, A^*, M^*) |I_j|^\theta \notag\\
& \leq \tfrac 12 \eta_0.
\end{align*}
	
\noi
In particular, with $j = J$, this implies that  there exists some $\eps > 0$ such that 
\begin{align*}
\|S(t-t_J) v(t_J) \|_{L^{q_d}_t([t_J, T^* + \eps];W_x^{1, r_d})}
& \leq  \eta_0.
\end{align*}

\noi
By further imposing that $|I_J|\leq \frac 12 \eta_0^\frac{2}{\theta}$,
we conclude from Lemma \ref{LEM:LWP4}
that the solution $v$ can be extended to $[0, T^* + \eps]$
for some $\eps > 0$,
which is a contradiction to the assumption $T^* < \infty$.
Therefore, if $T^* < \infty$, then we must have \eqref{Y3}.

\begin{remark}\label{REM:X^1}\rm
Suppose $T^* < \infty$.
Then, it follows from the argument above with Lemma~\ref{LEM:LWP4} and the subadditivity of the $X^1$-norm over disjoint intervals
(Lemma A.4 in \cite{BOP2}) that 
$v \in X^1([0, T^* - \dl))$
for any $\dl > 0$.
If $T^* = \infty$, we have 
$v \in X^1([0, T))$ for any finite $T > 0$.

\end{remark}

\section{Almost sure global well-posedness
of the defocusing energy-critical  NLS below the energy space}
\label{SEC:energy}

 In this section, we present the proof of Theorem \ref{THM:GWP}.
Namely, we prove almost sure global well-posedness
of the defocusing  energy-critical  NLS on $\R^d$, $d = 5, 6$:
\begin{equation}
\label{XNLS}
\begin{cases}
 i \dt u + \Dl  u =  |u|^\frac{4}{d-2} u,  \\
 u|_{t = 0} = \phi^\o,
\end{cases}
\qquad (t, x) \in \R\times \R^d.
\end{equation}

\noi
where
$\phi^\o$ is the  Wiener randomization 
of a given function $\phi \in H^s (\R ^d)$
for some $s < 1$.
As in Section \ref{SEC:LWP2}, 
we  consider the following Cauchy problem
for the defocusing NLS with a deterministic perturbation:
\begin{equation}
\begin{cases}
	 i \dt v + \Dl v =  |v + f|^\frac{4}{d-2}(v+f)\\
v|_{t = 0} = 0.
 \end{cases}
\label{ZNLS1}
\end{equation}

\noi
Under a suitable regularity assumption on $f$, 
Lemma \ref{LEM:LWP4} guarantees local existence of solutions
to \eqref{ZNLS1}.
In the following, we assume 
\begin{itemize}
\item[(i)]
$f$ is a linear solution $f  = S(t)\psi$
for some deterministic initial condition $\psi$, 
\item[(ii)] $f$ satisfies certain space-time integrability conditions.
\end{itemize}

\noi
Under these assumptions, 
we first establish crucial energy estimates (Proposition \ref{PROP:energy}
for $d = 6$ and Proposition \ref{PROP:energy2} for $d = 5$) 
for a solution $v$ to the perturbed
NLS \eqref{ZNLS1}.
This is the main new ingredient in this paper
as compared to \cite{BOP2}.
Once we have these energy estimates, 
we can proceed as in \cite{BOP2}
and hence we only sketch the argument.
Fix an interval $[0, T)$.
Given  $t_0 \in [0, T)$, we  iteratively apply the perturbation lemma (Lemma \ref{LEM:perturb})
on short time intervals $I_j = [t_j, t_{j+1}]$
and approximate a solution $v$ to the perturbed NLS \eqref{ZNLS1}
by the global  solution $w$ to the original NLS~\eqref{XNLS} with $w|_{t = t_0} = v(t_0)$.
This allows us to 
 show that the solution $v$ to the perturbed NLS \eqref{ZNLS1} 
exists on $[t_0, t_0 + \tau]$, where $\tau$ is independent of $t_0 \in [0, T)$
(Proposition~\ref{PROP:perturb2}).
By iterating this ``good'' local well-posedness, 
we can extend the solution $v$ to the entire interval $[0, T]$.
Since the choice of $T > 0$ was arbitrary, this shows that 
the perturbed NLS \eqref{ZNLS1} 
is globally well-posed.
In Subsection~\ref{SUBSEC:GWP}, 
we verify that 
the conditions imposed on $f$ for long time existence are satisfied with a large probability
by setting  $f(t) = z(t)  = S(t) \phi^\o$.
This yields Theorem \ref{THM:GWP}.

\subsection{Energy estimate for the perturbed NLS}

First, we discuss the following a priori\footnote{In Lemma \ref{LEM:mass} 
and Propositions \ref{PROP:energy} and \ref{PROP:energy2}, we prove 
a priori estimates
for a smooth solution $v$ with smooth $\psi$ and hence $f$.
By the standard argument via the local theory, 
one can show that these a priori estimates also hold for rough solutions as long as they exist.} control on the mass.

\begin{lemma}\label{LEM:mass}
Let $v$ be a solution
to \eqref{ZNLS1} with $f = S(t) \psi$.  Then, we have
\begin{align}
\int | v(t)|^2 dx \les \int | \psi |^2 dx,
\label{m1}
\end{align}

\noi
where the implicit constant is independent of $t \in \R$.
\end{lemma}

\begin{proof}
Note that $u = v + f$  satisfies \eqref{XNLS}.
Hence, by the mass conservation for \eqref{XNLS}, we have
\begin{align*}
\int |\psi|^2 dx  = \int | v(t) + f(t)|^2 dx = 
\int | v(t)|^2 dx + 2\Re \int v(t) \cj {f(t)} dx + \int  |f(t)|^2 dx.
\end{align*}

\noi
By the unitarity of the linear solution operator, we obtain 
\begin{align*}
\int | v(t)|^2 dx = -  2\Re \int v(t) \cj {f(t)} dx 
\leq \frac 12 \int | v(t)|^2 dx + 2\int | f(t)|^2 dx.
\end{align*}

\noi
By invoking the unitarity of the linear solution operator once again, 
we obtain \eqref{m1}.
\end{proof}

Next, we establish an energy estimate
when $d = 6$.
Recall the following conserved energy for NLS \eqref{XNLS}:
\begin{align*}
E(u) = \frac 12 \int |\nb u |^2 dx + \frac 13 \int |u|^3 dx .
\end{align*}

\noi
In the following, we estimate the growth of $E(v)$
for a solution $v$ to the perturbed NLS~\eqref{ZNLS1}.

\begin{proposition} \label{PROP:energy}
Let $d = 6$ and $s> \frac{20}{23}$.
Then, the  following energy estimate
holds
for a solution $v$ to the perturbed NLS \eqref{ZNLS1} with $f = S(t) \psi$:
\begin{align}
\dt E(v)(t) 
&\les 
\big(1+\|f(t)\|_{L^\infty_x}\big)E(v)(t)
+  \|f(t) \|_{L^6_x}^6\notag\\
& \hphantom{X}
+ \| f(t) \nb f(t)\|_{L^2_x}^2
+ \| v(t) \nb f(t)\|_{L^2_x}^2.
\label{energy1}
\end{align}

\noi
In particular, 
given $T > 0$, we have
\begin{align}
\sup_{t \in [0, T]} E(v)(t) 
\leq 
C\big(T, \|f \|_{A^s(T)}\big)
\label{energy2}
\end{align}

\noi
for any solution $v \in C([0, T]; H^1(\R^6))$ to the perturbed NLS \eqref{ZNLS1}
 with $f = S(t) \psi$, 
 where the $A^s(T)$-norm is defined by 
\begin{align*}
\|f \|_{A^s(T)}
:= \max\Big(
 \|\jb{\nb}^{s-} f \|_{L^\infty_{T, x}}, 
 \|f \|_{L^6_{T, x}}, 
\| f \|_{L^4_T W^{s, 4}_x}, 
\|   f \|_{L^4_T L^3_{x}}, 
 \|\psi\|_{L^2_x}, 
\|f\|_{Y^s_T}\Big).
\end{align*}

\end{proposition}

\begin{proof}

We first prove \eqref{energy1}.	
Since we work for fixed $t$, we suppress the $t$-dependence in the following.
Noting that $\dt (|v|^3) = 3|v| \Re(\cj v \dt v)$, we have 
\begin{align}
\dt E(v) 
& =  \underbrace{- \Re i \int \Dl v \Dl \cj v dx}_{ = 0}
+ \Re i \int |v + f| (v+ f) \Dl \cj vdx \notag \\
& \hphantom{X} +  \Re i \int \Dl v |v| \cj vdx
- \Re i \int |v+f|(v+f) |v|\cj vdx\notag \\
& =   \Re i \int \big\{|v + f| (v+ f) - |v| v\big\} \Dl \cj vdx
- \Re i \int |v+f|(v+f) |v|\cj vdx\notag \\
& =: \1 + \II.
\label{OO1}
\end{align}

\noi
By Young's inequality, we have
\begin{align}
\II 
& = - \underbrace{\Re i \int |v+f| |v|^3 dx}_{ = 0}
 - \Re i \int |v+f|\cdot f \cdot |v|\cj vdx\notag \\
&  \les (1 +  \|f\|_{L^\infty_{x}}) \int |v|^3 dx
 + \|f\|_{L^6_{x}}^6 \notag \\
& \les
(1 + \|f\|_{L^\infty_{x}}) E(v)
+ \|f\|_{L^6_{x}}^6.
\label{OO2}
\end{align}

Integrating by parts, we have 
\begin{align}
\1  =   - \Re i \int \nb \big\{|v + f| (v+ f) - |v| v\big\} \cdot \nb \cj vdx.
\label{OO2a}
\end{align}

\noi
Then, from \eqref{nonlin5},  \eqref{diff4}, 
and  \eqref{diff5}, 
we have
\begin{align}
\nabla \N(v + f) &  - \nb \N(v)\notag\\
& = \Re\big\{G(v + f)\nb (\cj{v+  f})
- G( v)\nb \cj{v}\big\}
+ H(v + f)\nb (v + f)
- H( v)\nb  v\notag\\
& = \Re\big\{G(v + f)\nb \cj f\big\}
+ \Re\big\{(G(v + f) - G(v)) \nb \cj v\big\}\notag\\
& \hphantom{X}
+ H(v + f)\nb f + 
(H(v+f) - H( v)) \nb  v
\label{OO3},
\end{align}

\noi
where $G(\zeta) = \frac{\zeta^2}{|\zeta|}$ and $H(\zeta) = |\zeta|$ are as in 
\eqref{diff4} and  \eqref{diff5} (with $d = 6$), respectively.
Let us denote by $\1_j$, $j = 1, \dots, 4$,  the contribution to $\1$ in \eqref{OO2a}
from the $j$th term on the right-hand side of \eqref{OO3}.

Proceeding as in \eqref{diff1}, we have 
\begin{align*}
G(v+f) - G(v) 
& = \int_0^1 \dd_\zeta G(v   + \theta f)\cdot f
 + \dd_{\cj \zeta}G(v + \theta f)\cdot \cj f  d\theta,\\
 H(v+f) - H(v) 
& = \int_0^1 \dd_\zeta H(v   + \theta f)\cdot f
 + \dd_{\cj \zeta}H(v + \theta f)\cdot \cj f  d\theta.
\end{align*}

\noi
Then, it follows from  \eqref{diff4} and  \eqref{diff5} that 
\begin{align}
\| G(v+f) - G(v) \|_{L^\infty_x} + 
\| H(v+f) - H(v) \|_{L^\infty_x}
\les \|f\|_{L^\infty_x}.
\label{OO3a}
\end{align}

\noi
Hence, from \eqref{OO2a}, \eqref{OO3},  and \eqref{OO3a}, we have
\begin{align}
| \1_2 + \1_4 |
\les \|f\|_{L^\infty_{x}} \|\nb v   \|_{L^2_x}^2
\les \|f\|_{L^\infty_{x}}  E(v).
\label{OO4}
\end{align}

\noi
Note that $|G(\zeta)| = |H(\zeta)|  = |\zeta|$.
Then, integrating by parts (in $x$), we have
\begin{align}
| \1_1 +\1_3 |
& \les   \|\nb v\|_{L^2_x}^2 
+  \|(v+f) \nb  f \|_{L^2_x}^2  \notag\\
& \les  E(v)
+   \|f \nb f  \|_{L^2_x}^2 
+\|v \nb  f \|_{L^2_x}^2.
\label{OO5}
\end{align}

\noi
Hence, \eqref{energy1} follows from 
\eqref{OO1}, \eqref{OO2}, \eqref{OO4}, and \eqref{OO5}.

Next, we discuss the second estimate \eqref{energy2}.
By solving the differential inequality \eqref{energy1} 
with $v|_{t = 0} = 0$
in a crude manner, we obtain
\begin{align}
 E(v)(\tau) 
 & \leq  
C \int_0^\tau e^{C(1+\|f\|_{L^\infty_{T, x}}) (\tau - t)}
\Big\{ 
 \|f (t)\|_{L^6_{x}}^6
+ \| f(t) \nb f(t)\|_{L^2_{x}}^2
+ \| v(t) \nb f(t)\|_{L^2_{x}}^2\Big\}dt\notag\\
 &  \leq 
C e^{C(1+\|f\|_{L^\infty_{T, x}}) T}
\Big\{  \|f \|_{L^6_{\tau, x}}^6
+ \| f \nb f\|_{L^2_{\tau, x}}^2
+ \| v \nb f\|_{L^2_{\tau, x}}^2\Big\}
\label{OO5a}
\end{align}

\noi
for any $\tau \in [0, T]$.
The estimate \eqref{OO5a} is by no means sharp.
It, however,  suffices for our purpose.

We can estimate $\| f \nb f\|_{L^2_{\tau, x}}$ as in the proof of Proposition \ref{PROP:nonlin1}.
Namely, by writing
\begin{align}
 \| f \nb f\|_{L^2_{\tau, x}}
\leq \sum_{N_1, N_2 \in 2^{\NB_0} } N_2
\| \P_{N_1}  f \P_{N_2} f\|_{L^2_{\tau, x}},
\label{OO5c}
\end{align}

\noi
we separate  the estimate
into two cases (i) $N_1 \ges N_2^\frac{1}{5}$
and (ii) $N_1 \ll N_2^\frac{1}{5}$.
Then, 
we can estimate the contribution from (i) by 
$ \| f \|_{L^4_\tau W^{s, 4}_x}^2$ for $ s > \frac 56$,
while we can apply 
Lemma \ref{LEM:biStv}
and estimate 
the contribution from (ii) 
by $\|f\|_{Y^s_\tau }^2$ for $ s > \frac 56$.
Hence, we obtain
\begin{align}
 \| f \nb f\|_{L^2_{\tau, x}}^2
\les  \| f \|_{L^4_\tau W^{s, 4}_x}^4
+ \|f\|_{Y^s_\tau}^4,
\label{OO5b}
\end{align}

\noi
provided that $s  > \frac 56$.

Next, we consider $\| v \nb f\|_{L^2_\tau L^2_x}$.
By writing
\[ \| v \nb f\|_{L^2_{\tau, x}}
\leq \sum_{N_1, N_2 \in 2^{\NB_0} } N_2
\| \P_{N_1}  v \P_{N_2} f\|_{L^2_{\tau, x}},
\]

\noi
we divide the argument into the following two cases:
\[ \text{(i)}\ N_1 \ges N_2^\g\qquad 
\text{and} \qquad 
\text{(ii)}\ 
N_1 \ll N_2^\g\]

\noi
for some $\g > 0$ (to be chosen later).
We first estimate 
the contribution from (i) $N_1 \ges N_2^\g$.
By interpolation and Lemma \ref{LEM:mass}, we have 
\begin{align}
 \sum_{\substack{N_1, N_2 \in 2^{\NB_0}\\N_1 \ges N_2^\g}} N_2
& \| \P_{N_1}  v \P_{N_2} f\|_{L^2_{\tau, x}}
 \les  \sum_{\substack{N_1, N_2 \in 2^{\NB_0}\\N_1 \ges N_2^\g}} 
 N_1^{1-} N_2^{1-\g+}
\| \P_{N_1}  v \P_{N_2} f\|_{L^2_{\tau, x}}\notag \\
& \les 
 \sum_{\substack{N_1, N_2 \in 2^{\NB_0}\\N_1 \ges N_2^\g}} 
\| \P_{N_1} \jb{\nb}^{1-} v \P_{N_2} \jb{\nb}^{s-} f\|_{L^2_{\tau, x}}\notag\\
&  
\leq C(T)  \Big\{ \sup_{t \in [0, \tau]} \big(E(v)(t)\big)^{\frac{1}{2}-}
\|  \psi \|_{ L^2_x}^{0+}
  + \| \psi\|_{L^2_x} \Big\}
\|  \jb{\nb}^{s-} f\|_{L^\infty_{\tau, x}}
\label{OO6}
\end{align}

\noi
for any $\tau \in [0, T]$, 
provided that
\begin{align}
s > 1 -\g.
\label{OO6a}
\end{align}

We now turn our attention to (ii) $N_1 \ll N_2^\g$.
Recall that $(q, r) = (2, 3)$ is admissible.
Hence, by Lemma \ref{LEM:biStv},  the Duhamel formula (with $v|_{t = 0} = 0$),
the linear estimate (Lemma~\ref{LEM:lin}) and 
the Strichartz estimates (Lemma \ref{LEM:Stv}), we have
\begin{align}
\| \P_{N_1}  v \P_{N_2} f\|_{L^2_{\tau, x}}
& \les N_1^{\frac{5}{2}-}N_2^{-\frac{1}{2}+}
\| \P_{N_1}  v \|_{Y^0_\tau}\|\P_{N_2} f\|_{Y^0_\tau }\notag\\
& \les N_1^{\frac{5}{2}-}N_2^{-\frac{1}{2}+}
\bigg\| \P_{N_1} \int_0^tS(t - t') |v+f|(v+f)(t') dt' \bigg\|_{Y^0_\tau}\|\P_{N_2} f\|_{Y^0_\tau}\notag\\
& \les N_1^{\frac{5}{2}-}N_2^{-\frac{1}{2}+}
\big(\|   v \|_{L^4_\tau L^3_{x}}^2
+ \|   f \|_{L^4_\tau L^3_{x}}^2\big)
\|\P_{N_2} f\|_{Y^0_\tau}
\label{OO7}
\end{align}

\noi
Fix  $\theta \in (0, 1)$ (to be chosen later).
We  apply \eqref{OO7} only to the $\theta$-power of the factor in 
\begin{align*}
 \sum_{\substack{N_1, N_2 \in 2^{\NB_0}\\N_1\ll N_2^\g}} N_2
\| \P_{N_1}  v \P_{N_2} f\|_{L^2_{\tau, x}}.
\end{align*}

\noi
Then, 
with \eqref{OO7}, we have
\begin{align*}
&  \sum_{\substack{N_1, N_2 \in 2^{\NB_0}\\N_1\ll N_2^\g}} 
 N_2 \| \P_{N_1}   v \P_{N_2} f\|_{L^2_{\tau, x}} \notag\\
&  \hphantom{XX}
\les 
 \sum_{\substack{N_1, N_2 \in 2^{\NB_0}\\N_1\ll N_2^\g}} 
N_2^{1 - \frac{\theta}{2}+}
\big(\|   v \|_{L^4_\tau L^3_{x}}^2
+ \|   f \|_{L^4_\tau L^3_{x}}^2\big)^\theta
\|\P_{N_2} f\|_{Y^0_\tau}^\theta
\| N_1^{\frac52 \frac{\theta}{1-\theta}-}\P_{N_1}  v \P_{N_2} f\|_{L^2_{\tau, x}}^{1-\theta}\notag\\
\intertext{By interpolation,}
&  \hphantom{XX}
 \les  \sum_{\substack{N_1, N_2 \in 2^{\NB_0}\\N_1\ll N_2^\g}} 
\big(\|   v \|_{L^4_\tau L^3_{x}}^2
+ \|   f \|_{L^4_\tau L^3_{x}}^2\big)^\theta
\|\P_{N_2} f\|_{Y^s_\tau}^\theta
 \| \P_{N_1}  v\|_{ L^2_{\tau, x}}^{1 -\frac 72\theta +}
\notag\\
 & \hphantom{XXXXXXXXXXXXXX}
 \times 
\| \P_{N_1} \jb{\nb} v\|_{L^2_{\tau, x}}^{\frac52 \theta-}
\| \P_{N_2} \jb{\nb}^{s-} f\|_{L^\infty_{\tau, x}}^{1-\theta},
\end{align*}

\noi
provided that 
\begin{align}
1 - \frac{\theta}{2}< s.
\label{OO9}
\end{align}

\noi
Summing over $N_1$ and $N_2$
and applying Lemma \ref{LEM:mass}, we obtain 
\begin{align}
\| v \nb f\|_{L^2_{\tau, x}}^2
\leq C(T,  \| f\|_{A^s(T)}) 
\Big\{ 1+ \sup_{t \in [0, \tau]} \big(E(v) (t)\big)^{1-}\Big\}
\label{OO9a}
\end{align}

\noi
for any $\tau \in [0, T]$, 
provided that 
\begin{align}
\frac 43 \theta +  \frac52\theta  < 1.
\label{OO10}
\end{align}

\noi
Optimizing \eqref{OO6a}, 
\eqref{OO9}, and  \eqref{OO10}, 
we obtain 
\[ s > \frac{20}{23}\]

\noi
with $\theta = \frac 6{23} -$
and $\g = 1-s +$.

Finally, 
putting 
\eqref{OO5a}, \eqref{OO5b}, \eqref{OO6}, and \eqref{OO9a}
together with $v|_{ t= 0} = 0$, we obtain
\begin{align*}
\sup_{t \in [0, \tau]} E(v)(t) 
\leq 
C(T, \|f \|_{A^s(T)})
\Big\{ 1 + 
\sup_{t \in [0, \tau]} \big(E(v) (t)\big)^{1-}
\Big\}
\end{align*}

\noi
for any $\tau \in [0, T]$.
Then, \eqref{energy2}
follows from  the standard continuity argument.
\end{proof}

We conclude this subsection by 
establishing an energy estimate
when $d = 5$.
As mentioned in Section \ref{SEC:1}, 
we study the growth of
the following modified energy:
\begin{align*}
\E(v) = \frac 12 \int |\nb v |^2 dx + \frac 3{10} \int |v+f|^\frac{10}{3} dx 
\end{align*}

\noi
for a solution $v$ to the perturbed NLS~\eqref{ZNLS1}.

\begin{proposition} \label{PROP:energy2}
Let $d = 5$ and $s> \frac{63}{68}$.
Then, the  following energy estimate
holds:
given $T > 0$, we have
\begin{align}
\sup_{t \in [0, T]} \E(v)(t) 
\leq 
C\big(T, \|f \|_{B^s(T)}\big)
\label{EE1}
\end{align}

\noi
for any solution $v \in C([0, T]; H^1(\R^5))$ to the perturbed NLS \eqref{ZNLS1}
 with $f = S(t) \psi$, 
 where the $B^s(T)$-norm is defined by 
\begin{align*}
\|f \|_{B^s(T)}
:= \max_{\substack{p = \frac{5}{2}, 3, 4\\q = 2, \frac{10}{3}}}\Big(
 \|\jb{\nb}^{s-} f \|_{L^\infty_{T, x}}, 
 \|\jb{\nb}^{s} f \|_{L^p_{T, x}}, 
\| f \|_{L^\frac{14}{3}_T L^\frac{10}{3}_{x}},
 \|f\|_{L^\infty_T L^q_x}, 
\|f\|_{Y^s_T}\Big).
\end{align*}

\end{proposition}

The proof of Proposition \ref{PROP:energy2}
is similar to that of Proposition \ref{PROP:energy}
but is more complicated
due to the (higher) fractional power of the nonlinearity.

\begin{proof}
Proceeding as in \eqref{OO1}
with  $\dt \big(|v+f|^\frac{10}{3}\big) = \frac{10}{3}|v+f|^\frac{4}{3} \Re\big((\cj{ v +f})\dt (v+f)\big)$, we have 
\begin{align}
\dt \E(v) 
& =  
 \Re i \int |v + f|^\frac{4}{3} (v+ f) \Dl \cj vdx \notag \\
& \hphantom{X} 
+  \Re i \int  (\Dl v+\Dl f)  |v+f|^\frac{4}{3}  (\cj {v+f})dx
- \underbrace{\Re i \int |v+f|^\frac{14}{3}dx}_{=0}\notag \\
& =  
  \Re i \int \nb\big(|v + f|^\frac{4}{3} (v+ f)\big) \cdot \nb \cj fdx \notag \\
\intertext{With \eqref{nonlin5},} 
& =
 \frac{4}{3} \Re i \int \frac{v+f}{|v+f|^\frac{2}{3} }\Re \big((v+f) \nb (\cj{v+f})\big) \cdot \nb \cj fdx \notag \\
& \hphantom{X} 
+   \Re i \int |v + f|^\frac{4}{3} \nb (v+ f) \cdot \nb \cj fdx \notag \\
& =
 \frac{5}{3} \Re i \int |v+f|^\frac{4}{3} \nb  v \cdot \nb \cj f dx 
+  \frac{2}{3} \Re i \int \frac{(v+f)^2}{|v+f|^\frac{2}{3} }\nb \cj v  \cdot \nb \cj fdx \notag \\
& \hphantom{X} 
+ \frac{2}{3} \Re i \int \frac{(v+f)^2}{|v+f|^\frac{2}{3} } \nb \cj f\cdot \nb \cj fdx \notag\\
& \les \|\nb v \|_{L^2_x}^2
+ \big\| |v+f|^\frac{4}{3} \nb f \big\|_{L^2_x}^2
+  \big\| |v+f|^\frac{4}{3} \nb  f\cdot \nb  f\big\|_{L^1_x}\notag\\
& \les \E(v) 
+ \big\| |v+f|^\frac{4}{3} \nb f \big\|_{L^2_x}^2
+  \big\| |v+f|^\frac{4}{3} \nb  f\cdot \nb  f\big\|_{L^1_x}.
\label{XX1}
\end{align}

\noi
By solving the differential inequality \eqref{XX1} 
with $v|_{t = 0} = 0$
in a crude manner, we obtain
\begin{align}
 \E(v)(\tau) 
 & \les  
 \int_0^\tau e^{C (\tau - t)}
\Big\{  \big\| |v+f|^\frac{4}{3} \nb f \big\|_{L^2_x}^2
+  \big\| |v+f|^\frac{4}{3} \nb  f\cdot \nb  f\big\|_{L^1_x}
\Big\}dt\notag\\
 &  \leq
 e^{C T}
\Big\{ \big\| |v+f|^\frac{4}{3} \nb f \big\|_{L^2_{\tau, x}}^2
+  \big\| |v+f|^\frac{4}{3} \nb  f\cdot \nb  f\big\|_{L^1_{\tau, x}}\Big\}\notag\\
& =: e^{CT} \big\{\1 + \II\big\}
\label{XX2}
\end{align}

\noi
for any $\tau \in [0, T]$,

We first consider $\1$.
By H\"older's inequality, we have
\begin{align}
 \big\| |v+f|^\frac{4}{3} \nb f \big\|_{L^2_{\tau, x}}
& \les  \big\| |f|^\frac{4}{3} \nb f \big\|_{L^2_{\tau, x}}
+ 
 \big\| |v|^\frac{4}{3} \nb f \big\|_{L^2_{\tau, x}}\notag\\
& \les \| f\|_{L^\infty_{\tau, x}}^\frac{1}{3} \| f \nb f \|_{L^2_{\tau, x}}
+ \| v\|_{L^\infty_\tau L^\frac{10}{3}_x}^\frac{1}{3}  
 \| v \nb f \|_{L^2_\tau L^\frac{5}{2}_x}.
\label{XX3}
\end{align}

\noi
Arguing as in \eqref{OO5c}, 
we have
\begin{align}
 \| f \nb f\|_{L^2_{\tau, x}}^2
\les  \| f \|_{L^4_\tau W^{s, 4}_x}^4
+ \|f\|_{Y^s_\tau}^4,
\label{XX4}
\end{align}

\noi
provided that $s  > \frac 45$.
On the other hand, by the dyadic  decomposition, we have 
\begin{align}
 \| v\|_{L^\infty_\tau L^\frac{10}{3}_x}^\frac{1}{3}  
\| v \nb f \|_{L^2_\tau L^\frac{5}{2}_x}
& \les \Big\{\sup_{t \in [0, \tau]} \big(\E(v)(t)\big)^\frac{1}{10} + 
 \| f\|_{L^\infty_\tau L^\frac{10}{3}_x}^\frac{1}{3}  \Big\}\notag \\
& \hphantom{XXX}
\times 
\sum_{N_1, N_2 \in 2^{\NB_0}}
N_2 \| \P_{N_1} v \P_{N_2} f \|_{L^2_\tau L^\frac{5}{2}_x}.
\label{XX5}
\end{align}

\noi
Then, by interpolation, we have\footnote{In 
the following, we drop  the summation over $N_1$ and $N_2$
for conciseness of the presentation.
Note that we can simply sum over $N_1$ and $N_2$ 
at the end  by losing
an $\eps$-amount of derivative.
Similar comments apply to other dyadic summations.
}
\begin{align}
N_2 \| \P_{N_1} v \P_{N_2} f \|_{L^2_\tau L^\frac{5}{2}_x} 
& \leq
N_2 \| \P_{N_1} v \P_{N_2} f \|_{L^2_{\tau, x}}^\frac{1}{2} 
\| \P_{N_1} v \P_{N_2} f \|_{L^2_\tau L^\frac{10}{3}_x}^\frac{1}{2} \notag \\
& \leq
N_2 \| \P_{N_1} v \P_{N_2} f \|_{L^2_{\tau, x}}^\frac{1}{2} 
\| \P_{N_1} v\|_{L^\infty_\tau L^\frac{10}{3}_x}^\frac{1}{2}
\| \P_{N_2} f \|_{L^2_\tau L^\infty_x}^\frac{1}{2} \notag \\
& \les 
\Big\{ \sup_{t \in [0, \tau]} \big(\E(v)(t)\big)^\frac{3}{20}
+ \| \P_{N_1} f\|_{L^\infty_\tau L^\frac{10}{3}_x}^\frac{1}{2}
\Big\}\notag\\
& \hphantom{XXX}
\times 
N_2 \| \P_{N_1} v \P_{N_2} f \|_{L^2_{\tau, x}}^\frac{1}{2} 
\| \P_{N_2} f \|_{L^2_\tau L^\infty_x}^\frac{1}{2}.
\label{XX6}
\end{align}

\noi
We now divide the argument into the following two cases:
\[ \text{(i)}\ N_1 \ges N_2^\g\qquad 
\text{and} \qquad 
\text{(ii)}\ 
N_1 \ll N_2^\g\]

\noi
for some $\g \in (0, 1)$ (to be chosen later).
We first estimate 
the contribution from (i) $N_1 \ges N_2^\g$.
By interpolation and Lemma \ref{LEM:mass}, 
we have
\begin{align}
N_2 \| \P_{N_1} v \P_{N_2} f \|_{L^2_{\tau, x}}^\frac{1}{2} 
& \| \P_{N_2}   f \|_{L^2_\tau L^\infty_x}^\frac{1}{2}
 \les N_1^{\frac 12-}
N_2^{1-\frac 12 \g +} \| \P_{N_1} v \P_{N_2} f \|_{L^2_{\tau, x}}^\frac{1}{2} 
\| \P_{N_2} f \|_{L^2_\tau L^\infty_x}^\frac{1}{2}\notag\\
& \les 
 \| \P_{N_1}  \jb{\nb}^{1-}v  \P_{N_2} \jb{\nb}^{s-} f \|_{L^2_{\tau, x}}^\frac{1}{2} 
\| \jb{\nb}^{s-} f \|_{L^2_\tau L^\infty_x}^\frac{1}{2}\notag\\
& \leq C(T) \Big\{ \sup_{t \in [0, \tau]} \big(\E(v)(t) \big)^{\frac{1}{4}-}
\|  \psi \|_{ L^2_x}^{0+}
+ \| \psi\|_{L^2_x}^\frac{1}{2}\Big\}
\|  \jb{\nb}^{s-} f\|_{L^\infty_{\tau, x}}
\label{XX7}
\end{align}

\noi
for any $\tau \in [0, T]$, 
provided that
\begin{align}
s > 1 -\tfrac 12 \g.
\label{XX8}
\end{align}

Next, we consider  (ii) $N_1 \ll N_2^\g$.
Recall that $(q, r) = \big(2, \frac{10}{3} \big)$ is admissible.
Then, proceeding as in \eqref{OO7}
with  Lemma \ref{LEM:biStv},  the Duhamel formula (with $v|_{t = 0} = 0$),
the linear estimate (Lemma \ref{LEM:lin}) and 
the Strichartz estimates (Lemma \ref{LEM:Stv}), we have
\begin{align}
\| \P_{N_1}  v \P_{N_2} f\|_{L^2_{\tau, x}}
& \les N_1^{2-}N_2^{-\frac{1}{2}+}
\| \P_{N_1}  v \|_{Y^0_\tau}\|\P_{N_2} f\|_{Y^0_\tau}\notag\\
& \les N_1^{2-}N_2^{-\frac{1}{2}+}
\big(\|  v \|_{L^\frac{14}{3}_\tau L^\frac{10}{3}_{x}}^\frac{7}{3}
+ \| f \|_{L^\frac{14}{3}_\tau L^\frac{10}{3}_{x}}^\frac{7}{3}\big)
\|\P_{N_2} f\|_{Y^0_\tau}
\label{XX9}
\end{align}

\noi
As in the proof of Proposition \ref{PROP:energy}, 
we   apply \eqref{XX9} only to the $\theta$-power for some $\theta \in (0, 1)$.
With \eqref{XX9}, we have
\begin{align}
& 
N_2 \| \P_{N_1} v \P_{N_2} f \|_{L^2_\tau L^2_x}^\frac{1}{2} 
\| \P_{N_2}   f \|_{L^2_\tau L^\infty_x}^\frac{1}{2}
 \notag\\
&  \hphantom{XX}
\les N_2^{1 - \frac{\theta}{4}+}
\big(\|  v \|_{L^\frac{14}{3}_\tau L^\frac{10}{3}_{x}}^\frac{7}{3}
+ \| f \|_{L^\frac{14}{3}_\tau L^\frac{10}{3}_{x}}^\frac{7}{3}\big)^{\frac{1}{2}\theta}
 \|\P_{N_2} f\|_{Y^0_\tau}^{\frac{1}{2}\theta}
\notag\\
 & \hphantom{XXXXXXXXXXXXX}
 \times 
\| N_1^{ \frac{2\theta}{1-\theta}-}\P_{N_1}  v \P_{N_2} f\|_{L^2_{\tau, x}}^{\frac{1}2(1-\theta)}
\| \P_{N_2}   f \|_{L^2_\tau L^\infty_x}^\frac{1}{2}
\notag\\
\intertext{By interpolation and Lemma \ref{LEM:mass},}
&  \hphantom{XX}
 \les C(T)
\big(\|  v \|_{L^\frac{14}{3}_\tau L^\frac{10}{3}_{x}}^\frac{7}{3}
+ \| f \|_{L^\frac{14}{3}_\tau L^\frac{10}{3}_{x}}^\frac{7}{3}\big)^{\frac{1}{2}\theta}
  \|\P_{N_2} f\|_{Y^s_\tau}^{\frac 12 \theta}
 \| \P_{N_1}  f\|_{ L^2_{\tau, x}}^{\frac{1-3\theta}{2}}
\notag\\
 & \hphantom{XXXXXXXXXXXXX}
 \times 
\| \P_{N_1} \jb{\nb} v\|_{L^\infty_\tau L^2_{x}}^{\theta}
\| \P_{N_2} \jb{\nb}^{s-} f\|_{L^\infty_{\tau, x}}^{1-\theta}
\label{XX10}
\end{align}

\noi
for any $\tau \in [0, T]$, 
provided that 
\begin{align}
1 - \frac{\theta}{4} < s.
\label{XX11}
\end{align}

Hence, from \eqref{XX3}, \eqref{XX4}, \eqref{XX5}, \eqref{XX6}, \eqref{XX7}, and \eqref{XX10}, 
we obtain
\begin{align}
\1 =  \big\| |v+f|^\frac{4}{3} \nb f \big\|_{L^2_{\tau, x}}^2
\leq C(T,  \| f\|_{B^s(T)})
\sup_{t \in [0, \tau]} \Big\{ 1+ \big(\E(v) (t)\big)^{1-}\Big\}
\label{XX12}
\end{align}

\noi
for any $\tau \in [0, T]$, 
provided that 
\begin{align*}
\frac{1}{2} + 
\frac 7{10} \theta + \theta < 1.
\end{align*}

\noi
In particular, by choosing $\theta = \frac 5{17}-$
and $\g = \frac 12 \theta$, 
it follows from \eqref{XX8} and \eqref{XX11} that 
the estimate \eqref{XX12} holds for
\begin{align}
s > \frac{63}{68} \approx 0.9265.
\label{XX13}
\end{align}

Next, we estimate $\II$ in \eqref{XX1}.
By symmetry, we have
\begin{align*}
\II & =  \big\| |v+f|^\frac{4}{3} \nb  f\cdot \nb  f\big\|_{L^1_{\tau, x}}
\les
\sum_{\substack{N_2, N_3 \in 2^{\NB_0}\\N_2 \geq N_3}}
N_2 N_3  \big\| |v+f|^\frac{4}{3}  \P_{N_2} f \cdot \P_{N_3}  f\big\|_{L^1_{\tau, x}}\notag\\
& \les
\sum_{\substack{N_2, N_3 \in 2^{\NB_0}\\N_2 \geq N_3}}
N_2 N_3  \big\| |v|^\frac{4}{3}  \P_{N_2}  f\cdot \P_{N_3}  f\big\|_{L^1_{\tau, x}}\notag\\
& \hphantom{XXXXX}
+ 
\sum_{\substack{N_2, N_3 \in 2^{\NB_0}\\N_2 \geq N_3}}
N_2 N_3  \| f \|_{L^2_{\tau, x}}^\frac{1}{3} \| f  \P_{N_2} f\|_{L^2_{\tau, x}} \|  \P_{N_3}  f\|_{L^3_{\tau, x}}\notag\\
& =: \II_1 + \II_2.
\end{align*}

\noi
We first estimate $\II_2$.
By the dyadic decomposition, we have 
\begin{align*}
\II_2 
& = 
\sum_{\substack{N_1 N_2, N_3 \in 2^{\NB_0}\\N_2 \geq N_3}}
N_2 N_3  \| f \|_{L^2_{\tau, x}}^\frac{1}{3} \| \P_{N_1} f  \P_{N_2} f\|_{L^2_{\tau, x}} \|  \P_{N_3}  f\|_{L^3_{\tau, x}}
\notag\\
& \leq \sum_{\substack{N_1 N_2, N_3 \in 2^{\NB_0}\\N_2 \geq N_3}}
N_2^{2-2s}  \| f \|_{L^2_{\tau, x}}^\frac{1}{3} \| \P_{N_1} f  \P_{N_2} \jb{\nb}^{s}f\|_{L^2_{\tau, x}} 
\|  \P_{N_3}  \jb{\nb}^s f\|_{L^3_{\tau, x}}\notag\\
& \leq C(T, \|f\|_{B^s(T)})
\sum_{N_1 N_2 \in 2^{\NB_0}}
N_2^{2-2s+}   \| \P_{N_1} f  \P_{N_2} \jb{\nb}^{s}f\|_{L^2_{\tau, x}} 
\end{align*}

\noi
for any $\tau \in [0, T]$.
If $N_1 \ges N_2^\g$ for some $\g \in (0, 1)$, 
then we have 
\begin{align}
N_2^{2-2s+}   \| \P_{N_1} f  \P_{N_2} \jb{\nb}^{s}f\|_{L^2_{\tau, x}} 
& \les N_1^{0-}
 N_2^{2-2s-\g s+ }   \| \P_{N_1} \jb{\nb}^s f  \P_{N_2} \jb{\nb}^{s}f\|_{L^2_{\tau, x}} \notag\\
& \les 
\| f\|_{L^4_\tau W_x^{s, 4}}^2,
 \label{XX15}
\end{align}

\noi
provided that $2 - 2s - \g s  < 0$, namely 
\begin{align}
 s > \frac{2}{2+\g}.
 \label{XX16}
\end{align}

\noi
If $N_1 \ll N_2^\g$, 
then by applying Lemma \ref{LEM:biStv}, 
we have
\begin{align}
N_2^{2-2s+}   \| \P_{N_1} f  \P_{N_2} \jb{\nb}^{s}f\|_{L^2_{\tau, x}} 
& \les N_1^{2-s -}
 N_2^{\frac 32-2s+}   \| \P_{N_1} f \|_{Y^s_\tau}
 \| \P_{N_2} f \|_{Y^s_\tau}
\notag\\
& \ll
 N_2^{\frac 32-2s+ \g (2 - s)+}   \| \P_{N_1} f \|_{Y^s_\tau}
 \| \P_{N_2} f \|_{Y^s_\tau}
\notag\\
& \ll
 \|  f \|_{Y^s_\tau}^2, 
 \label{XX17}
\end{align}

\noi
provided that $\frac 32-2s+ \g (2 - s) < 0$,
namely
\begin{align}
 s > \frac{3 + 4\g}{4+2\g}.
 \label{XX18}
\end{align}

\noi
It follows from 
 \eqref{XX15} and \eqref{XX17} with 
\eqref{XX16} and \eqref{XX18}
that 
\begin{align}
\II_2 
 \leq C(T, \|f\|_{B^s(T)})
\label{XX19a}
\end{align}

\noi
for any $\tau \in [0, T]$,
provided that 
\begin{align}
 s > \frac89 \approx 0.8889.
 \label{XX19}
\end{align}

Finally, we estimate $\II_1$.
By H\"older's inequality, we have
\begin{align}
\II_1 
& \les   
\sum_{\substack{N_2, N_3 \in 2^{\NB_0}\\N_2 \geq N_3}}
N_2^{2-2s}  \| v\|_{L^\frac{10}{3}_{\tau, x}}^\frac{1}{3} \| v  \P_{N_2} \jb{\nb}^{s-}  f\|_{L^2_{\tau, x}}
\| \P_{N_3} \jb{\nb}^s  f\|_{L^\frac{5}{2}_{\tau, x}}.
\label{XX20}
\end{align}

\noi
In the following, we estimate
\[
\| v  \P_{N_2} \jb{\nb}^{s-}  f\|_{L^2_{\tau, x}}
\les \sum_{N_1\in 2^{\NB_0}}
\|\P_{N_1} v  \P_{N_2} \jb{\nb}^{s-}  f\|_{L^2_{\tau, x}}.
\]

\noi
If $N_1 \ges N_2^\g$ for some $\g \in (0, 1)$, 
then 
\begin{align}
 N_2^{2-2s+} \| v  \P_{N_2} \jb{\nb}^{s-}  f\|_{L^2_{\tau, x}}
& \les N_1^{1-} N_2^{2-2s-\g +} \| v  \P_{N_2} \jb{\nb}^{s-}  f\|_{L^2_{\tau, x}}\notag\\
& \les C(T)  \| \jb{\nb}  v\|_{L^\infty_\tau L^2_x} \|  \jb{\nb}^{s-}  f\|_{L^\infty_{\tau, x}}
\label{XX21}
\end{align}

\noi
for any $\tau \in [0, T]$, provided that 
$2 - 2s < \g < 1$.

If $N_1 \ll N_2^\g$, then 
by applying \eqref{XX9} to the $\theta$-power
$\| v  \P_{N_2} \jb{\nb}^{s-}  f\|_{L^2_{\tau, x}}$
 as before, we have
\begin{align}
 N_2^{2-2s+} \| v &  \P_{N_2} \jb{\nb}^{s-}  f\|_{L^2_{\tau, x}}\notag\\
& \les  N_2^{2-2s-\frac 12 \theta+} 
\big(\|  v \|_{L^\frac{14}{3}_\tau L^\frac{10}{3}_{x}}^\frac{7}{3}
+ \| f \|_{L^\frac{14}{3}_\tau L^\frac{10}{3}_{x}}^\frac{7}{3}\big)^\theta
\|\P_{N_2} f\|_{Y^s_\tau}^\theta
\notag\\
& \hphantom{XXXXXXXX}
\times 
\| N_1^\frac {2\theta-}{1-\theta} \P_{N_1} v \|_{L^2_{\tau, x}}^{1-\theta}
\|  \P_{N_2} \jb{\nb}^{s-}  f\|_{L^\infty_{\tau, x}}^{1-\theta}\notag\\
\intertext{By interpolation and Lemma \ref{LEM:mass},}
& \leq  C(T) N_2^{2-2s-\frac 12 \theta+} 
\big(\|  v \|_{L^\frac{14}{3}_\tau L^\frac{10}{3}_{x}}^\frac{7}{3}
+ \| f \|_{L^\frac{14}{3}_\tau L^\frac{10}{3}_{x}}^\frac{7}{3}\big)^\theta
\|\P_{N_2} f\|_{Y^s_\tau}^\theta
\|  f \|_{L^\infty L^2_x}^{1-3\theta+}
\notag\\
& \hphantom{XXXXXXXX}
\times 
\|  \P_{N_1} \jb{\nb} v \|_{L^2_{\tau, x}}^{2\theta-}
\|  \P_{N_2} \jb{\nb}^{s-}  f\|_{L^\infty_{\tau, x}}^{1-\theta}\notag\\
& \leq C(T, \| f\|_{B^s(T)}) 
\big(1 + \|   v \|_{L^\infty_\tau L^\frac{10}{3}_{x}}^{\frac{7}{3}\theta}\big)
\|  \jb{\nb} v \|_{L^\infty_\tau L^2_{x}}^{2\theta-}
\label{XX22}
\end{align}

\noi
for any $\tau \in [0, T]$, 
provided that 
\begin{align}
s > 1 - \frac{\theta}{4}.
\label{XX23}
\end{align}

\noi
Putting \eqref{XX20}, 
\eqref{XX21}, 
and \eqref{XX22} together, 
we obtain 
\begin{align}
\II_1
\leq C(T,  \| f\|_{B^s(T)})
\sup_{t \in [0, \tau]} \Big\{ 1+ \big(\E(v) (t)\big)^{1-}\Big\}
\label{XX24}
\end{align}

\noi
by choosing $\theta \in (0, 1)$ such that $\frac{1}{10} + \frac{7}{10}\theta + \theta < 1$
with $\g = \frac{\theta}{2}$.
In particular, by choosing $\theta = \frac{9}{17}-$, 
the regularity restriction \eqref{XX23}
yields
\begin{align}
s > \frac{59}{68}\approx 0.8676.
\label{XX25}
\end{align}

Therefore, 
it follows from 
\eqref{XX2}, 
\eqref{XX12},   \eqref{XX19a}, and \eqref{XX24} with \eqref{XX13}, \eqref{XX19},  and  \eqref{XX25}
that 
\begin{align*}
\sup_{t \in [0, \tau]} \E(v)(t) 
\leq 
C(T, \|f \|_{B^s(T)})
\Big\{ 1 + 
\sup_{t \in [0, \tau]} \big(\E(v) (t)\big)^{1-}
\Big\}
\end{align*}

\noi
for any $\tau \in [0, T]$, 
provided that $s > \frac{63}{68}$.
Therefore, \eqref{EE1}
follows from  the standard continuity argument.
\end{proof}

\subsection{Long time existence of solutions to the perturbed NLS}
\label{SUBSEC:stability}

Our main goal in this subsection is to prove long time existence of solutions to the perturbed NLS \eqref{ZNLS1}
under some regularity assumptions on the perturbation $f$
(Proposition \ref{PROP:perturb2}).
The main ingredients are
the energy estimates (Propositions \ref{PROP:energy} and \ref{PROP:energy2})
and the following perturbation lemma.

\begin{lemma}[Perturbation lemma]\label{LEM:perturb}
Given $d = 5$ or $6$, let $(q_d, r_d)$ be the admissible pair in~\eqref{qr1}.
Let  $I$ be
a compact interval
with $|I|\leq 1$.
Suppose that $v \in C(I; H^1(\R^d))$ satisfies the following perturbed NLS:
\begin{align*}
i \dt v + \Dl v = |v|^\frac{4}{d-2} v + e,
\end{align*}

\noi
satisfying
\begin{align*}
\|v \|_{L^{q_d}_t(I; W_x^{1, r_d}(\R^d))}
 + \|v\|_{L^\infty(I;  H^1(\R^d))} \leq R
\end{align*}

\noi
for some $R \geq 1$.
Then, there exists $\eps_0 = \eps_0(R) > 0$
such that
if
we have
\begin{align*}
\|w_0 - v(t_0) \|_{  H^1(\R^d)}
+\|e\|_{  N^1(I)} \leq \eps
\end{align*}

\noi
for
some  $w_0 \in H^1(\R^d)$,
some $t_0 \in I$, and some $\eps < \eps_0$, then
there exists a solution
$w \in X^1(I)\cap
C(I; H^1(\R^d))$
to  the defocusing  NLS  \eqref{XNLS}
with $w(t_0) = w_0$
such that
\begin{align*}
\|w\|_{X^1(I)}+\|v\|_{X^1(I)} & \leq C(R), \\
\|w - v\|_{X^1(I)} & \leq C(R)\eps,
\end{align*}

\noi
where $C(R)$ is a non-decreasing function of $R$.
\end{lemma}

See
\cite{CKSTT,  TV, TVZ} for perturbation and stability results
on the usual Strichartz and Lebesgue spaces.
For  perturbation lemmas
involving the critical $X^{1}$-norm,
see \cite{IP, BOP2}.
The proof of Lemma \ref{LEM:perturb}
follows from a straightforward modification
of the proof  of Lemma 7.1 in \cite{BOP2}
and hence we omit details.

We now state a long time existence result 
for the perturbed NLS \eqref{ZNLS1}.
Fix $d = 5$ or $6$
and let $s \in (s_*, 1)$, where $s_*$ is as in Theorem \ref{THM:GWP}.
Then, let $\dl = \dl(d, s) > 0$ be as in Corollary \ref{COR:nonlin2}.
Given $T>0$,
suppose that $f \in \wt E_{M, T}$ for some $M > 0$, 
where $ \wt E_{M, T}$ is as in~\eqref{ER2}.
Namely, we have
\begin{equation}
 \| f \|_{Y^s([0, T))}
+ \| f \|_{S^s ([0, T) )} \le M. 
\label{P0}
 \end{equation}

\noi
Then,  Lemma \ref{LEM:LWP4}
guarantees
existence of a solution $v\in C([0,\tau_0];H^1(\T^d) )\cap X^1([0,\tau_0])$ to the perturbed NLS \eqref{ZNLS1},
at least for a short time $\tau_0 > 0$.
Furthermore, assume that there exists $K > 0$
such that 
\begin{align}
\text{(i)} \  \|f\|_{A^s(T)} \leq K
\ \text{ when $d = 6$}
\qquad \text{and}\qquad 
\text{(ii)}  \ \|f\|_{B^s(T)} \leq K \ \text{ when $d = 5$},
\label{P1}
\end{align}

\noi
where $A^s(T)$ and $B^s(T)$ are as in 
Propositions \ref{PROP:energy} and \ref{PROP:energy2}.
Then, it follows from Lemma~\ref{LEM:mass} and Propositions \ref{PROP:energy}
and \ref{PROP:energy2} that  there exists $R = R(K, T) >0$
such that 
\begin{align}
\| v\|_{L^\infty([0, T]; H^1(\R^d))} \leq R
\label{Xenergy}
\end{align}

\noi
for a  solution
$v$  to  \eqref{ZNLS1}.

Under these assumptions, 
by iteratively applying Lemma \ref{LEM:perturb}, 
we obtain the following long time existence result
for the perturbed NLS \eqref{ZNLS1} on $[0, T]$.

\begin{proposition}\label{PROP:perturb2}
Let $d = 5, 6$
and $s \in (s_*,  1)$, where
$s_*$ is as in Theorem \ref{THM:GWP}.
Given $T>0$, 
assume that the hypotheses \eqref{P0} and \eqref{P1} hold.
Then, 
there exists
$\tau = \tau(R,M, T, s)>0$ 
such that,  given any $t_0 \in [0, T)$,
the solution $v$ to \eqref{ZNLS1}
exists on $[t_0, t_0 + \tau]\cap [0, T]$.
In particular, the energy estimate \eqref{Xenergy}
 guarantees existence of $v$ on
the entire interval $[0, T]$.

\end{proposition}

 Proposition \ref{PROP:perturb2}
follows from a straightforward modification 
of the proof of  Proposition~7.2 in \cite{BOP2}.
Hence,  we omit the details of the proof
but we briefly describe the main idea in the following.
Given $t_0 \in [0, T)$, 
the main idea is 
to  approximate  a solution $v$ to the perturbed NLS \eqref{ZNLS1}
by the global solution $w$ to the original NLS~\eqref{XNLS} with $w|_{t = t_0} = v(t_0)$
on $[t_0, t_0 + \tau]$, where $\tau = \tau(R,M, T, s)>0$
 is independent of $t_0 \in [0, T)$.
We achieve this goal 
by iteratively applying the perturbation lemma (Lemma \ref{LEM:perturb})
on short time intervals.
This is possible thanks to 
(i)~the a priori control \eqref{P0} and \eqref{Xenergy}
on $f$ and  the $H^1$-norm  of $v(t)$, respectively,  on $[0, T]$
and (ii) the following space-time control   on 
the global solution $w$ to \eqref{XNLS} due to Vi\c{s}an \cite{V}:
\[ \| w \|_{L^{\frac{2(d+2)}{d-2}}_{t, x}(\R\times \R^d)} \leq C(\| v(t_0)\|_{H^1}) = C(R).\]

\noi
See  the proof of  Proposition~7.2 in \cite{BOP2} for details.
In the following, we  point out 
the difference between the assumptions  in Proposition \ref{PROP:perturb2} above
and those in Proposition~7.2 in \cite{BOP2}.
The assumption in \cite{BOP2}
would read as 
``$\| f \|_{S^s (I)} \le |I|^\beta$ 
for any interval $I \subset [0, T]$'' in our context.
Note that  we are making a weaker assumption 
 on the $S^s$-norm  in \eqref{P0}.
This is possible thanks to the appearance of the factor $|I|^\theta$ 
in the nonlinear estimate \eqref{Y1} in Corollary \ref{COR:nonlin2}.
Namely, in this paper, 
we already exploited the subcritical nature of the perturbation
and created the factor  $|I|^\theta$ in \eqref{Y1}.
Compare this with Lemma 6.2 in \cite{BOP2}.

\subsection{Proof of Theorem \ref{THM:GWP}}	
\label{SUBSEC:GWP}

In this subsection, we present the proof of Theorem \ref{THM:GWP}.
By Borel-Cantelli lemma, it suffices to prove the following 
 ``almost'' almost sure global existence result.
See \cite{CO, BOP2}.

\begin{proposition}\label{PROP:asGWP}
Let $d = 5, 6$ and $s \in (s_*,  1)$,
where $s_*$ is as in Theorem \ref{THM:GWP}.
Given $\phi \in H^s(\R^d)$, let $\phi^\o$ be its Wiener randomization defined in \eqref{rand}.
Then, given any $T, \eps > 0$, there exists a set $\wt \O_{T, \eps}\subset \O$
such that
\begin{itemize}
\item[\textup{(i)}]
$P(\wt \O_{T, \eps}^c) < \eps$,

\item[\textup{(ii)}]
For each $\o \in \wt \O_{T, \eps}$, there exists a (unique) solution $u$
to \eqref{NLS}  on $[0, T]$
with $u|_{t = 0} = \phi^\o$.

\end{itemize}

\end{proposition}

The proof of Proposition \ref{PROP:asGWP} is
analogous
to that of Proposition 8.1 in \cite{BOP2}.
The main difference appears in the definitions of $\O_2$
and $\O_3$
below, incorporating the energy estimate~\eqref{Xenergy}
and the simplified assumption \eqref{P0}.

\begin{proof}
Fix  $T, \eps > 0$.
Set $M = M(\eps, \|\phi\|_{H^s})$ by 
\begin{equation*}	
M  \sim \|\phi\|_{H^s}\Big(\log\frac{1}{\eps}\Big)^\frac{1}{2}.
\end{equation*}

\noi
Without loss of generality, we assume that $\eps > 0$ is sufficiently small
such that $M = M(\eps, \|\phi\|_{H^s}) \geq 1$.
Defining $ \O_1 = \O_1(\eps)$
by
\[  \O_1 := \big\{ \o \in \O:\, \| \phi^\o \|_{H^s} \leq M\big\},\]

\noi
it follows from
 Lemma \ref{LEM:prob1} that
\begin{equation}
 P( \O_1^c) < \frac{\eps}{3}.
\label{asGWP1}
\end{equation}

Given $K > 0$, 
define $\O_2 = \O_2(T, K)$ by 
\[  \O_2 := \big\{ \o \in \O:\,
\|S(t) \phi^\o\|_{F^s(T)} \leq K \big\},\]

\noi
where $F^s(T) = A^s(T)$ when $ d= 6$
and $= B^s(T)$ when $d = 5$.
Then, by Lemmas~\ref{LEM:prob1} and~\ref{LEM:prob11}, 
we can choose $K = K(T, \eps, \|\phi\|_{H^s}) \gg 1 $ such that 
\begin{align}
 P( \O_2^c) < \frac{\eps}{3}.
\label{asGWP1a}
\end{align}

\noi
Hence, the energy estimate \eqref{Xenergy} holds
with some $R = R(K, T) = R(T, \eps) > 0$.

Now, let  $\tau = \tau(R, M, T, s )$ 
be as in Proposition \ref{PROP:perturb2}.
Let  $\dl = \dl(d, s) > 0$ be as in Corollary~\ref{COR:nonlin2}
and set $q = \frac{4}{ 1- 4\dl}$.
With $I_j = [j \tau_*, (j + 1) \tau_*] $
 for some $\tau_* \leq \tau$ (to be chosen later),
 we  partition the interval $[0, T]$ as
\[ [0, T] = \bigcup_{j = 0}^{[\frac T{ \tau_*}]} I_j \cap [0, T]\]

\noi
and  define $\O_3$ by
\[ \O_3 := \Big\{ \o \in \O:
\| S(t) \phi^\o \|_{S^s(I_j)}
\leq M , \, j = 0, \dots, \big[\tfrac T{\tau_*}\big] \Big\}.
\]

\noi
Then,
by Lemma \ref{LEM:prob11}
and taking 
$\tau_* = \tau_*(T, \eps, \|\phi\|_{H^s})>0$ sufficiently small,
we have
\begin{align}
P(\O_3^c)
& \leq  \sum_{j = 0}^{[\frac{T}{\tau_*}] }
P\Big( \| S(t) \phi^\o  \|_{S^s(I_j)}> M\Big)
  \leq C \frac{T}{\tau_*}  \exp\bigg(-c \frac{M^2}{\tau_*^\frac{2}{q} \|\phi\|_{H^s}^2}\bigg)\notag \\
&  \leq C \frac{T}{\tau_\ast}\cdot \tau_* \exp\bigg(-c \frac{M^2}{2\tau_*^\frac{2}{q} \|\phi\|_{H^s}^2}\bigg)
\leq C  T \exp\bigg(-\frac{c}{2\tau_*^\frac{2}{q} \|\phi\|_{H^s}^2}\bigg)\notag\\
& < \frac{\eps}{3}.
\label{asGWP4}
\end{align}

Finally, set $\wt \O_{T, \eps} :=  \O_1 \cap\O_2\cap \O_3$.
Then, from \eqref{asGWP1}, \eqref{asGWP1a}, and \eqref{asGWP4}, we conclude that 
\[ P(\wt \O_{T, \eps}^c) < \eps.\]

\noi
Moreover, for $\o \in \wt \O_{T, \eps}$,
 we can iteratively apply  Proposition \ref{PROP:perturb2} 
and construct the solution $v = v^\o$ 
to \eqref{NLS1a}
on each $[j \tau_*, (j+1)\tau_*]$, $j = 0, \dots, [\frac T\tau_*]-1$,
and $\big[ [\frac T \tau_*]\tau_*, T\big]$.
This completes the proof of Proposition \ref{PROP:asGWP}.
\end{proof}

\section{Probabilistic construction of finite time blowup solutions below the energy space}
\label{SEC:7}

In this section, we present the proof of Theorem \ref{THM:4}.
We first recall the following definition of a weak solution to \eqref{NLS5}.
See \cite{IW}.

\begin{definition}\label{DEF:weak}\rm
We say that $v$ is a weak solution to \eqref{NLS5} on $[0, T)$
if $v$ belongs to $L^\frac{d+2}{d-2}_\text{loc}([0, T) \times \R^d)$ and 
satisfies
\begin{align}
\intt _{[0,T] \times \R^{d}}  v\cdot (-i \dt \psi + \Dl \psi )\, dx dt 
= i \al \int_{\R^d} v_0\cdot  \psi (0) \,dx 
+ \ld  \intt _{[0,T] \times \R^{d}}   |v+\eps z|^\frac{d+2}{d-2} \cdot \psi\,  dx dt
\label{weak1}
\end{align}

\noi
for any test function\footnote{By convention, our test function $\psi$ 
has compact support but 
does not have to vanish at $t = 0$.
The same comment applies to the test function $\eta = \eta(t)$ below.} $\psi \in C_c^{\infty}([0, T) \times \R^d)$.

\end{definition}

Fix $v_0 \in H^1(\R^d)$.
Then, for any $\al > 0$ and $\eps > 0$, 
Proposition \ref{PROP:LWP3} establishes almost sure local well-posedness of 
the following Duhamel formulation:
\begin{align}
v(t) = \al S(t) v_0 -i \ld \int_0^t S(t - t') |v + \eps z^\o|^\frac{d+2}{d-2}(t') dt'.
\label{NLS6}
\end{align}

The following lemma shows that the solution $v$ 
to \eqref{NLS6} is indeed a weak solution to \eqref{weak1}.

\begin{lemma}\label{LEM:weak2}
Let $d=5,6$ and  $1-\frac{1}{d}<s<1$.
Given $\phi \in H^s (\R ^d)$, let $\phi^{\omega}$ be its Wiener randomization defined in \eqref{rand}
and let $z^\o = S(t) \phi^\o$.
Then, given any   $v_0 \in H^1(\R^d)$, $\al > 0$, $\eps > 0$, and $T>0$, 
any local-in-time solution $v \in C([0, T);H^1(\R^d)) \cap X^1([0, T))$ to the Duhamel formulation \eqref{NLS6}
is almost surely a weak solution on $[0, T)$ in the sense of Definition~\ref{DEF:weak}.

\end{lemma}

We first present the proof of Theorem \ref{THM:4}, 
assuming Lemma \ref{LEM:weak2}.
We prove Lemma \ref{LEM:weak2} at this end of this section.
Note that while 
Proposition \ref{PROP:LWP3} guarantees 
the  existence of the solution $v$ to \eqref{NLS6}
at least for some small $T_\o>0 $, 
Lemma \ref{LEM:weak2} 
{\it assumes} its existence on $[0, T)$ for some given $T > 0$.

In the following, we only consider \eqref{NLS2} with $\ld =1$ 
and assume that $v_0$ satisfies \eqref{A2}.
The proof of Theorem \ref{THM:4} is based on the so-called test function method \cite{Z1, Z2}
and we closely follow the argument in \cite{II15}.
We first define two test functions
$\eta = \eta(t) \in C_c^\infty([0, \infty); [0, 1])$
and $\theta = \theta(x) \in C^\infty_c(\R^d; [0, 1])$
such that 
\[\eta(t) = 
\begin{cases}
1 & \text{for }0 \leq t < \frac 12,\\
0 & \text{for } t\geq 1, 
\end{cases}
\qquad \text{and}\qquad
\theta(x) = 
\begin{cases}
1 & \text{for }0 \leq |x| < \frac 12,\\
0 & \text{for } |x|\geq 1.
\end{cases}
\]

\noi
We also define the scaled test functions
$\eta_{_T}$ and $\theta_{_T}$ by 
$\eta_{_T}(t) := \eta\big(\frac{t}{_{T}})$
and $\theta_{_T}(x) := \big(\frac{x}{\sqrt{T}}\big)$.
Finally, we set 
$\psi_{_T}(t, x) := \eta_{_T}(t) \theta_{_T}(x)$.

Given $T \geq 1$, let $v = v^\o\in X^1([0, T))$ be a solution to the Duhamel formulation \eqref{NLS6}
on $[0, T)$.  Define $\1$ and $\II$ by 
\begin{align}
\1(T) = \intt_{[0, T) \times B_{\sqrt{T}}} |v+\eps z^\o|^p \cdot \psi_{_T}^\l \, dx dt
\qquad \text{and}\qquad 
\II(T) = \Im \intt_{B_{\sqrt{T}}} v_0\cdot  \theta_{_T}^\l \, dx ,
\label{weak1a}
\end{align}

\noi	
where $B_r$ denotes the ball of radius $r$ centered at 0 in $\R^d$
and $\l \in \NB$ such that $\l \geq 2p' + 1$.
Here, $p' = \frac{d+2}{4}$ denotes the H\"older conjugate of  $p = \frac{d+2}{d-2}$.
By Lemma \ref{LEM:weak2}
and taking the real part of the weak formulation \eqref{weak1}, 
we obtain 
\begin{align}
 \1(T) - \al \II(T)
& = \Im \intt _{[0,T] \times \R^{d}}  v \cdot \dt \psi^\l_T \, dx dt 
+ \Re \intt _{[0,T] \times \R^{d}}  v \cdot  \Dl \psi^\l_T \, dx dt \notag \\
& =:  \III_1(T) + \III_2(T).
\label{weak2}
\end{align}

By  $\l - 1\geq \frac{\l}{p}$ and the triangle inequality, we have
\begin{align}
\III_1(T) 
& \les T^{-1}   \intt_{[0, T) \times B_{\sqrt{T}}} |v| \cdot 
 \eta_{_T}^{\l-1}\theta_{_T}^\l
  \eta'\big(\tfrac{t}{T}\big)\, dx dt  \notag \\
& \les T^{-1}   \intt_{[0, T) \times B_{\sqrt{T}}} |v+\eps z^\o| \cdot 
 \psi_{_T}^\frac{\l}{p}\, dx dt 
+ \eps  T^{-1}   \intt_{[0, T) \times B_{\sqrt{T}}} |z^\o|  
\, dx dt \notag\\
& \les T \big(\1(T)\big)^\frac{1}{p} 
+ \eps  T^{-1}  \|z^\o\|_{L^1_{t, x}([0, T) \times B_{\sqrt{T}})}.
\label{weak3}
\end{align}

\noi
A similar computation with    $\l - 2\geq \frac{\l}{p}$ and the triangle inequality yields
\begin{align}
\III_2(T) 
& \les T \big(\1(T)\big)^\frac{1}{p} 
+ \eps  T^{-1}  \|z^\o\|_{L^1_{t, x}([0, T) \times B_{\sqrt{T}})}.
\label{weak4}
\end{align}

From \eqref{weak2}, \eqref{weak3}, and \eqref{weak4}
with Young's inequality,  we have
\begin{align}
 - \al \II(T)
&  \leq -  \1(T) + 
C  T \big(\1(T)\big)^\frac{1}{p} 
+ C\eps  T^{-1}  \|z^\o\|_{L^1_{t, x}([0, T) \times B_{\sqrt{T}})}\notag\\
&  \leq -  \1(T) + 
C'  T^{p'}  +  \1(T) 
+ C\eps  T^{-1}  \|z^\o\|_{L^1_{t, x}([0, T) \times B_{\sqrt{T}})}\notag \\
&  \leq 
C'  T^\frac{d+2}{4}  
+ C\eps  T^{-1}  \|z^\o\|_{L^1_{t, x}([0, T) \times B_{\sqrt{T}})}.
\label{weak5}
\end{align}

\noi
On the other hand, from \eqref{A2} and \eqref{weak1a}
with a change of variables, we have
\begin{align}
- \II (T) \geq  T^\frac{d-k}{2} L(T) : = 
T^\frac{d-k}{2} \int_{B_\frac{1}{\sqrt{T}}}|x|^{-k} \theta^\l\,dx.
\label{weak6}
\end{align}

Given $R > 0$,  $T\geq 1$, and $\eps  >0$, define the set $\O_{R, \eps}$
by 
\[ \O_{R,  \eps} := \big\{ \omega \in \Omega : \| \phi^\o\|_{L^2_{ x}} \leq \eps^{-1} R\big\}.\]

\noi
Then, it follows from Lemma \ref{LEM:prob1} that 
\begin{align*}
P(\O_{R,  \eps}^c)\leq C\exp \bigg( - c \frac{R^2}{\eps^2 \|\phi\|_{L^2}^2}\bigg).
\end{align*}

\noi
In particular, 
$P(\O_{R,  \eps}^c) \to 0$ as $R \to \infty$ or $\eps \to 0$
(while keeping the other fixed).

Then, putting   \eqref{weak5} and   \eqref{weak6} together with $T\geq 1$,  we obtain
\begin{align}
  \al 
  & \leq CL^{-1}(T)\Big\{ T^\frac{-d+2k +2}{4} 
+ \eps  T^\frac {-d + k - 2}{2} \|z^\o\|_{L^1_{t, x}([0, T) \times B_{\sqrt{T}})}\Big\}\notag \\
  & \leq CL^{-1}(T)\Big\{ T^\frac{-d+2k +2}{4} 
+ \eps  T^\frac {-d + 2k+ 2 }{4} \|z^\o\|_{L^\infty_t L^2_{x}([0, T) \times B_{\sqrt{T}})}\Big\} \notag\\
  & \leq CL^{-1}(T) T^\frac{-d+2k +2}{4} (1 + R)
\label{weak8}
\end{align}

\noi
for $ \o \in \O_{R,  \eps}$.
In the following, we fix $R > 0$ and $\eps > 0$
and work on $\O_{R,  \eps}$.
Namely, the following argument holds uniformly in 
 $ \o \in \O_{R,  \eps}$
 and we suppress the dependence on $\o$.

Suppose that given $\al > 0$, 
the maximal existence time $T^*(\al) \geq 4$.
Since $k < d$, we have $ L(4) <\infty$.
In particular, by setting $T = 4$ in \eqref{weak8}, 
we obtain 
\[ \al \les 1 + R.\]

\noi
This in turn implies that there exists $\al_0 = \al_0(R)>0$
such that $T^*(\al) < 4$ for all $\al \geq  \al_0$.

Fix  $\al > \al_0$.
Then, 
 by noting that $L(T)$ defined in \eqref{weak6} is decreasing on $[0, \infty)$, 
we conclude  from \eqref{weak8} that 
\begin{align*}
  \al 
   \leq CL^{-1}(4) T^\frac{-d+2k +2}{4} (1 + R)
      \leq C(R) T^\frac{-d+2k +2}{4}
\end{align*}

\noi
for any $0 < T \leq T^*(\al) < 4$.
Hence, we obtain 
the following upper bound on the maximal time of existence:
\begin{align*}
T^*(\al) \leq C'(R) \alpha^{\frac{4}{-d+2k +2}}.
\end{align*}

\noi
Lastly, \eqref{A6} follows from the blowup alternative \eqref{Y3}.
This proves  Theorem~\ref{THM:4}.

We conclude this paper by presenting the proof of Lemma \ref{LEM:weak2}.
While the proof is standard, we include it for completeness.

\begin{proof}[Proof of Lemma \ref{LEM:weak2}]

Write the solution $v$ to \eqref{NLS6} on $[0, T)$
as 
\[v(t) = \al S(t) v_0 -i \ld \I[\N(v+\eps z^\o)](t),\]

\noi
where $\I$ is as in \eqref{duhamel}
and $\N(u) = |u|^\frac{d+2}{d-2}$.
First, we show that the linear part $\al S(t) v_0$ satisfies
\begin{align}\label{L1}
\intt _{[0,T] \times \R^{d}}  v\cdot (-i \dt \psi + \Dl \psi )\, dx dt 
= i \al \int_{\R^d} v_0\cdot  \psi (0) \,dx 
\end{align}

\noi
for any test function $\psi \in C_c^{\infty}([0, T) \times \R^d)$.

Let $v_{0, n}$ be smooth functions 
converging to $v_0$ in $H^1(\R^d)$.
Then, $\al S(t) v_{0, n}$, $n \in\mathbb{N}$,  solves the linear Schr\"odinger equation:
$i\dt v + \Dl v = 0$ and is smooth on $[0, T)\times \R^d$.
Integrating by parts, we have
\begin{align} \label{L2}
\intt _{[0,T] \times \R^{d}}  \al S(t) v_{0, n}\cdot (-i \dt \psi + \Dl \psi )\, dx dt 
=  i \al \int_{\R^d} v_{0, n}\cdot  \psi (0) \,dx. 
\end{align}

\noi
By H\"older's  inequality and the unitarity of $S(t)$ on $L^2(\R^d)$, we have 
\begin{align*}
\bigg|\intt _{[0,T] \times \R^{d}}\al (S(t) v_0 &  - 
S(t) v_{0, n}) (-i \dt \psi + \Dl \psi )\, dx dt \bigg| \notag \\
& \les \| v_0 - v_{0, n}\|_{L^2}
\big( \|\psi\|_{W^{1, 1}_T L^2_x} + \|\psi\|_{L^1_T H^2_x}\big)
\too 0.  
\end{align*}

\noi
Similarly, the right-hand side of \eqref{L2} converges
to the right-hand side of \eqref{L1} as $n \to \infty$.
Hence, \eqref{L1} holds.

Next, we consider the nonlinear part $-i \ld \I(v+\eps z^\o)$.
Let $v_n$ be smooth functions on $ [0, T)\times \R^d$
converging to $v$ in $X^1([0, T))$.
Then, by Proposition \ref{PROP:nonlin1} with 
Lemmas \ref{LEM:prob1} and~\ref{LEM:prob11},
we have
\begin{align} \label{ZZ1}
\big\| \I[\N(v+\eps z^\o)] -  \I[\N(v_n+\eps z^\o)]\big\|_{C_TH^1}\too 0,
\end{align}

\noi
\noi
almost surely.
Let $w_n = -i \ld \I[\N(v_n+\eps z^\o)]$.
Then, $w_n$ is the smooth solution to  the following inhomogeneous linear Schr\"odinger equation:
\[\begin{cases}
i\dt w_n + \Dl w_n = \ld   |v_n+\eps z|^\frac{d+2}{d-2}\\
w_n |_{t = 0} = 0.
\end{cases}
\]

\noi
Then, proceeding as above  with \eqref{ZZ1}
and integrating by parts, we have 
\begin{align}
\intt _{[0,T] \times \R^{d}}
 & -i \ld \I(v+\eps z^\o)\cdot (-i \dt \psi + \Dl \psi )\, dx dt 
 =  \lim_{n \to \infty}
 \intt _{[0,T] \times \R^{d}}
 w_n\cdot (-i \dt \psi + \Dl \psi )\, dx dt \notag\\
& =  \lim_{n \to \infty} 
 \intt _{[0,T] \times \R^{d}}
 (i \dt w_n + \Dl w_n ) \cdot \psi\, dx dt \notag\\
& = \lim_{n \to \infty} \ld  \intt _{[0,T] \times \R^{d}}
   |v_n+\eps z|^\frac{d+2}{d-2} \cdot \psi \, dx dt
= \ld  \intt _{[0,T] \times \R^{d}}
   |v+\eps z|^\frac{d+2}{d-2} \cdot \psi \, dx dt
   \label{ZZ2}
\end{align}

\noi
for any test function $\psi \in C_c^{\infty}([0, T) \times \R^d)$.
Hence, the weak formulation \eqref{weak1} follows from \eqref{L1} and \eqref{ZZ2}.
This completes the proof of Lemma \ref{LEM:weak2}.
\end{proof}

\begin{acknowledgment}

\rm 
T.O.~was supported by the European Research Council (grant no.~637995 ``ProbDynDispEq'').
M.O.~
was supported by JSPS KAKENHI Grant number JP16K17624.

\end{acknowledgment}

\end{document}